\documentclass[twoside,10pt]{article}
\usepackage{geometry}
\usepackage{graphicx}
\usepackage{indentfirst}                                
\usepackage[colorlinks]{hyperref}
\hypersetup{linkcolor=blue,filecolor=black,urlcolor=blue, citecolor=black}   
\usepackage[numbers,sort&compress]{natbib}         
\usepackage[perpage,symbol*,marginal]{footmisc}   

\usepackage{amsmath}                                
\usepackage{amsfonts}
\usepackage{amssymb}                                
\usepackage{bm}                                            
\usepackage{mathrsfs}                                 
\usepackage{amsthm}                                   
\usepackage{amsxtra}
\usepackage{pifont}
\usepackage{authblk}
\allowdisplaybreaks[4]

\ifxetex
\usepackage{letltxmacro}
\LetLtxMacro\SavedIncludeGraphics\includegraphics
\def\includegraphics#1#{
	\IncludeGraphicsAux{#1}%
}%
\newcommand*{\IncludeGraphicsAux}[2]{%
	\XeTeXLinkBox{%
		\SavedIncludeGraphics#1{#2}%
}}
\fi
\newcommand\orcidicon[1]{\href{https://orcid.org/#1}{\includegraphics[scale=0.02]{orcid.pdf}}}

\usepackage[titletoc,title]{appendix}
\usepackage{titlesec}
\titleformat{\section}{\large\bfseries\centering}{\thesection}{1em}{}
\titleformat{\subsection}{\centering\bfseries}{\thesubsection}{0.5em}{}
\titleformat{\subsubsection}[runin]{\bfseries}{\thesubsubsection.}{0.4em}{}[.]

\numberwithin{equation}{section}
\newtheorem{lemma}{Lemma}[section]
\newtheorem{proposition}[lemma]{Proposition}
\newtheorem{theorem}{Theorem}[section]

\newtheorem{corollary}[lemma]{Corollary}
\newtheorem{remark}{Remark}[section]

\usepackage{color}
\definecolor{DarkRed}{RGB}{139,0,0}
\definecolor{Purple}{RGB}{128,0,128}
\usepackage{ulem}

	\usepackage{accents}
	\makeatletter
	\def\wideubar{\underaccent{{\cc@style\underline{\mskip10mu}}}}
	\def\Wideubar{\underaccent{{\cc@style\underline{\mskip8mu}}}}
	\makeatother
	\makeatletter
	\def\widebar{\accentset{{\cc@style\underline{\mskip10mu}}}}
	\def\Widebar{\accentset{{\cc@style\underline{\mskip8mu}}}}
	\makeatother
	\newcommand{\VERTiii}[1]{{\left\vert\kern-0.3ex\left\vert\kern-0.3ex\left\vert #1
			\right\vert\kern-0.3ex\right\vert\kern-0.3ex\right\vert}}
	\newcommand{\VERT}{\vert\kern-0.3ex\vert\kern-0.3ex\vert}
	\newcommand{\VERTl}{\left\vert\kern-0.3ex\left\vert\kern-0.3ex\left\vert}
	\newcommand{\VERTr}{\right\vert\kern-0.3ex\right\vert\kern-0.3ex\right\vert}
	\newcommand{\VERTbig}{\big\vert\kern-0.3ex\big\vert\kern-0.3ex\big\vert}
	\newcommand{\VERTBig}{\Big\vert\kern-0.3ex\Big\vert\kern-0.3ex\Big\vert}
\begin{document}
		
		\title{\bf The Binary–ternary Boltzmann Equation Near Equilibrium
		}
		

		 \author[a]{Xianshen Hu\thanks{Email: xianshen.hu@connect.polyu.hk}}
		 \author[b,c]{Linjie Xiong\thanks{Corresponding author. Email: xlj@hnu.edu.cn}}
		  \author[b]{Zhihong Xu\thanks{Email: zhihongxu@hnu.edu.cn}}
		 \affil[a]{\small Department of Applied Mathematics, The Hong Kong Polytechnic University, Hong Kong, China}
		 \affil[b]{\small School of Mathematics, Hunan University, Changsha 410082, China}
		 \affil[c]{\small Hunan Provincial Key Laboratory of Intelligent Information Processing and Applied Mathematics, Hunan University, Changsha 410082, China}

		\date{\empty}
		\maketitle

		\begin{abstract}
			 The classical Boltzmann equation, formulated by Boltzmann in 1872, describes the statistical behavior of a dilute gas in a non-equilibrium state. While the global well-posedness of the classical Boltzmann equation has been extensively studied and is relatively well-understood, this model becomes inadequate when the gas is no longer sufficiently dilute. Therefore, it is necessary to investigate higher-order extensions of the Boltzmann equation that account for multi-particle interactions. Recently, Ampatzoglou and Pavlović provided the rigorous derivation of both the ternary [\textit{Comm. Math. Phys.} \textbf{387} (2021), no. 2, 793–863.] and the binary-ternary [\textit{Forum Math. Sigma} \textbf{13} (2025), Paper No. e52, 95 pp.] Boltzmann equations from a classical particle system, yielding an explicit form of the ternary collision operator. However, the global existence near equilibrium for the binary-ternary Boltzmann equation remains an interesting and fundamental question, a gap that this paper seeks to address as its primary contribution.
		\end{abstract}
		
		\vspace*{2mm}
		{\small
			
			\vspace*{2mm}
			{
				\tableofcontents
			}
			
			
			
			
			

		}
		
		\bigskip
		
		\section{Introduction}\label{sec:intro}
		The binary-ternary Boltzmann transport equation under consideration governs the time evolution of a particle distribution function \( f(t,x,v)\ge0 \) in non-equilibrium dilute gases and is expressed as:
		\begin{align}\label{BT-f}
			\left\{
			\begin{aligned}
				&\partial_t f + v \cdot \nabla_x f = Q_{B}(f,f) + Q_{T}(f,f,f), \quad (t,x,v) \in (0,\infty) \times \mathbb{R}_{x}^3 \times \mathbb{R}_{v}^3,\\[0.5mm] 
				&f(0,x,v) = f_0(x,v), \quad (x,v) \in \mathbb{R}_{x}^3 \times \mathbb{R}_{v}^3.\\
			\end{aligned}
			\right.
		\end{align}  
	This kinetic equation describes the statistical behavior of a system of particles undergoing both binary and ternary collisions. The binary collision operator \( Q_{B}(f,f) \) (quadratic) and the ternary operator \( Q_{T}(f,f,f) \) (cubic) model two-body and three-body interactions, respectively.  
	The initial condition \( f_0 \ge 0 \) specifies the phase-space distribution at time \( t = 0 \). 
	When the term \( Q_{T}(f,f,f) \) on the right-hand side of equation \eqref{BT-f} is absent, the system reduces to the classical Boltzmann equation, which describes the statistical behavior of a dilute gas in a non-equilibrium state \cite{MR1307620,Glassey1996}. 
	However, in many realistic scenarios, the gas is sufficiently dense that higher-order collisions occur more frequently, thereby exerting a significant influence on its evolution. Consequently, such interactions must be taken into account.  
	Recently, the work of Ampatzoglou-Pavlović provided the rigorous derivation of both the ternary Boltzmann equation \cite{MR4337060,MR4315662} and a binary-ternary Boltzmann equation \cite{MR4874187} from a classical system of particles, yielding the explicit form of the ternary collision operator \( Q_{T}(f,f,f) \). This extends the classical Boltzmann framework to systems such as colloids, in which higher-order interactions play a critical role in the dynamics beyond the dilute-gas regime governed solely by binary collisions. 
	This generalized transport equation offers a broader framework for modeling gas dynamics where multi-body interactions are non-negligible, particularly in intermediate to high-density regimes.	
	Moreover, it is noteworthy that Ampatzoglou, Pavlović, and Warner derived the arbitrarily high-order Boltzmann equation for hard spheres in a subsequent work \cite{ampatzoglouWarner}.

		The binary collision operator \( Q_B(f,g) \) in \eqref{BT-f} is defined as  
		\begin{equation}\label{Q_B(f,g)}
			Q_B(f,g) = \int_{\mathbb{S}^2 \times \mathbb{R}^3} B_2(u, \omega) \left( f'g'_1 - fg_1 \right) d\omega dv_1,  
		\end{equation} 
		where \( \omega \in \mathbb{S}^2 \) parametrizes the collision direction, \( u = v_1 - v \) is the pre-collisional relative velocity, and \( f' \equiv f(v') \), \( g'_1 \equiv g(v'_1) \), \( f \equiv f(v) \), \( g_1 \equiv g(v_1) \). The post-collisional velocities \( v', v'_1 \) are determined by the binary collision law:  
		\begin{equation}\label{binary collision law}
			\begin{cases}
				v' = v + (\omega \cdot u)\omega, \\
				v'_1 = v_1 - (\omega \cdot u)\omega,
			\end{cases} 
		\end{equation} 
		satisfying momentum-energy conservation:  
		\begin{align}\label{b conservation} 
			\begin{aligned}
				v' + v'_1 &= v + v_1,\\[0.5mm]
				|v'|^2 + |v'_1|^2 &= |v|^2 + |v_1|^2. 
			\end{aligned} 
		\end{align}  
		For the post-collisional relative velocity \( u' = v'_1 - v' \), the binary micro-reversibility condition \( u' \cdot \omega = -u \cdot \omega \) and the conservation \( |u'| = |u| \) hold. 
		Given \( \omega \in \mathbb{S}^{2} \), the linear measure-preserving involution \( T_\omega: \mathbb{R}^{6} \to \mathbb{R}^{6} \), defined by \((v, v_1) \mapsto (v', v_1')\) via the binary collision law \eqref{binary collision law}, satisfies \( T_\omega^{-1} = T_\omega \) and \( |\det T_\omega| = 1 \), ensuring reversibility and phase-space volume conservation (see \cite{Glassey1996}, Lemma 1.6.1).
		
		The factor \( B_2(u, \omega) \) in the integrand of \eqref{Q_B(f,g)} is referred to as the “binary interaction differential cross-section”. It depends on the relative velocity \( u = v_1 - v \) and the direction vector \( \omega \in \mathbb{S}^2 \), and quantifies the transition probability of binary elastic collisions. This cross-section is assumed to take the separable form:
		\begin{equation}\label{B_2}
			B_2(u, \omega) = |u|^{\gamma_{2}} \, b_2(\hat{u} \cdot \omega), \quad \gamma_{2} \in (-3,1],
		\end{equation}
		where \( \hat{u} = u / |u| \) denotes the unit vector in the direction of the relative velocity.
		The binary angular function \( b_2: [-1,1] \to [0,+\infty) \) is an even function, and the relations \( u' \cdot \omega = -u \cdot \omega \) and \( |u'| = |u| \) imply the invariance 
		$$ B_2(u',\omega) = B_2(u,\omega), \quad \forall u \neq 0,\ \omega \in \mathbb{S}^2.$$  
		Under the cut-off assumption, \( b_2(\hat{u} \cdot \omega) \in L^1(\mathbb{S}^2) \) holds for any fixed unit vector \( \hat{u} \), and we define the finite normalization constant:  
		\begin{equation}\label{cutoff b2}
			0 < b := \int_{\mathbb{S}^2} b_2(\hat{u} \cdot \omega) \, d\omega < \infty.
		\end{equation} 
		The parameter range \( \gamma_{2} \in (0, 1] \) corresponds to hard potentials, while \( \gamma_{2} \in (-3, 0) \) characterizes soft potentials, with \( \gamma_{2} = 0 \) representing Maxwellian molecules. Notably, the classical hard sphere model arises as a special case of \eqref{B_2} when \( \gamma_{2} = 1 \) and the angular cross-section takes the form \( b_2(z) = \frac{|z|}{2} \) (refer to Remark 2.2. in \cite{MR4412071}).

		The ternary collisional operator \( Q_T(f, g, h) \), first introduced in \cite{MR4337060,MR4315662} and considered here in its symmetrized form as presented in Section 5.4 of \cite{MR4337060}, is expressed as:
		\begin{align}\label{Q_T(f, g, h)} 
			\begin{aligned}
				Q_T(f, g, h) =  &\int_{\mathbb{S}^{5} \times \mathbb{R}^{6}}B_3(\bm{u}, \bm{\omega}) \left( f^* g_1^* h_2^* - f g_1 h_2 \right) d\bm{\omega} dv_{1}dv_{2} \\[0.5mm]
				 &+ 2 \int_{\mathbb{S}^{5} \times \mathbb{R}^{6}} B_3(\bm{u}_1, \bm{\omega}) \left( f^{1*} g_1^{1*} h_2^{1*} - f g_1 h_2 \right) d\bm{\omega} dv_{1}dv_{2},
			\end{aligned} 
		\end{align}   
		where \( \bm{\omega} = \begin{pmatrix} \omega_1 \\ \omega_2 \end{pmatrix} \in \mathbb{S}^{5} \) denotes the vector of impact directions, and \( \bm{u} = \begin{pmatrix} v_1 - v \\ v_2 - v \end{pmatrix} \), \( \bm{u}_1 = \begin{pmatrix} v - v_1 \\ v_2 - v_1 \end{pmatrix} \) represent the relative velocity vectors for central and adjacent ternary interactions, respectively.  
		When the tracked particle is central, the collisional formulas are
		\begin{equation}\label{ternary collision law 1}
			\begin{cases} 
				v^* = v + \dfrac{\bm{u} \cdot \bm{\omega}}{1 + \omega_1 \cdot \omega_2} (\omega_1 + \omega_2), \\[3mm] 
				v_1^* = v_1 - \dfrac{\bm{u} \cdot \bm{\omega}}{1 + \omega_1 \cdot \omega_2} \omega_1, \\[3mm] 
				v_2^* = v_2 - \dfrac{\bm{u} \cdot \bm{\omega}}{1 + \omega_1 \cdot \omega_2} \omega_2.
			\end{cases} 
		\end{equation}  
		When the tracked particle is adjacent, the  corresponding formulas are:  
		\begin{equation}\label{ternary collision law 2}
			\begin{cases} 
				v^{1*} = v - \dfrac{\bm{u}_1 \cdot \bm{\omega}}{1 + \omega_1 \cdot \omega_2} \omega_1, \\[3mm] 
				v_1^{1*} = v_1 + \dfrac{\bm{u}_1 \cdot \bm{\omega}}{1 + \omega_1 \cdot \omega_2} (\omega_1 + \omega_2), \\[3mm] 
				v_2^{1*} = v_2 - \dfrac{\bm{u}_1 \cdot \bm{\omega}}{1 + \omega_1 \cdot \omega_2} \omega_2. 
			\end{cases} 
		\end{equation}  
		Both configurations satisfy momentum-energy conservation:  
		\begin{align}\label{t conservation} 
			\begin{aligned}
				v^* + v_1^* + v_2^* = v^{1*} + &v_1^{1*} + v_2^{1*} = v + v_1 + v_2,\\[0.5mm]
				|v^*|^2 + |v_1^*|^2 + |v_2^*|^2 = |v^{1*}|^2 + &|v_1^{1*}|^2 + |v_2^{1*}|^2 = |v|^2 + |v_1|^2 + |v_2|^2, 
			\end{aligned} 
		\end{align}   
		and the ternary micro-reversibility conditions:  
		\[
		\bm{u}^* \cdot \bm{\omega} = -\bm{u} \cdot \bm{\omega}, \quad \bm{u}_1^{1*} \cdot \bm{\omega} = -\bm{u}_1 \cdot \bm{\omega}.
		\]  
		For \( \bm{\omega} \in \mathbb{S}^{5} \), the transformations \( T_{1,\bm{\omega}}: (v, v_1, v_2) \mapsto (v^*, v_1^*, v_2^*) \) and \( T_{2,\bm{\omega}}: (v, v_1, v_2) \mapsto (v^{1*}, v_1^{1*}, v_2^{1*}) \) are linear measure-preserving involutions on \( \mathbb{R}^{9} \) (see Proposition 2.3 of \cite{MR4315662}). 
		 Specifically, they satisfy:  
		 \begin{enumerate}
		 	\item[$\bullet$] Involution property: \( T_{1,\bm{\omega}}^{-1} = T_{1,\bm{\omega}} \) and \( T_{2,\bm{\omega}}^{-1} = T_{2,\bm{\omega}} \); 
		 	\item[$\bullet$] Volume preservation: \( |\det T_{1,\bm{\omega}}| = 1 \) and \( |\det T_{2,\bm{\omega}}| = 1 \).  
		 \end{enumerate}
		The symmetric velocity norm \( |\tilde{\bm{u}}| \), defined as  
		\[
		|\tilde{\bm{u}}| := \left( |v - v_1|^2 + |v - v_2|^2 + |v_1 - v_2|^2 \right)^{1/2},  
		\]  
		quantifies the collective relative motion of three interacting particles and satisfies the inequality 
		\begin{equation}\label{tildeu-u}
			\frac{1}{\sqrt{3}} |\tilde{\bm{u}}| \leq |\bm{u}|, |\bm{u}_1| \leq |\tilde{\bm{u}}|.
		\end{equation} 
		Conservation of momentum and energy ensures the invariance of this norm under ternary collisions, i.e., 
		$$|\tilde{\bm{u}}^*| = |\tilde{\bm{u}}^{1*}| = |\tilde{\bm{u}}| ,$$
		 where \( |\tilde{\bm{u}}^*| \) and \( |\tilde{\bm{u}}^{1*}| \) represent post-collisional configurations. 
		 Normalizing the relative velocities via  
		\[
		\bar{\bm{u}} := |\tilde{\bm{u}}|^{-1} \bm{u}, \quad \bar{\bm{u}}_1 := |\tilde{\bm{u}}|^{-1} \bm{u}_1,  
		\]  
		confines \( \bar{\bm{u}} \) and \( \bar{\bm{u}}_1 \) to the 5-dimensional ellipsoid  
		\[
		\mathbb{E}_1 = \left\{ (\nu_1, \nu_2) : |\nu_1|^2 + |\nu_2|^2 + |\nu_1 - \nu_2|^2 = 1 \right\}.
		\]  
		The ternary interaction differential cross-section \( B_3(\bm{u}, \bm{\omega}) \), governing the transition probability of ternary collisions, is defined as  
		\begin{align}\label{B_3}
			\begin{aligned}
				&B_3(\bm{u}, \bm{\omega}) = |\tilde{\bm{u}}|^{\gamma_3} \, b_3\left( \bar{\bm{u}} \cdot \bm{\omega}, \omega_1 \cdot \omega_2 \right),\\[0.5mm]
				&B_3(\bm{u}_{1}, \bm{\omega}) = |\tilde{\bm{u}}|^{\gamma_3} \, b_3\left( \bar{\bm{u}}_{1} \cdot \bm{\omega}, \omega_1 \cdot \omega_2 \right), 
			\end{aligned} 
		\quad\gamma_3 \in (-6, 1]. 
		\end{align}    
		The angular component \( b_3: [-1, 1] \times [-1/2, 1/2] \to [0, \infty) \) is a non-negative function satisfying symmetry (\( b_3(-z, w) = b_3(z, w) \)), micro-reversibility (\( b_3(\bar{\bm{u}}^* \cdot \bm{\omega}, \omega_1 \cdot \omega_2) = b_3(\bar{\bm{u}} \cdot \bm{\omega}, \omega_1 \cdot \omega_2) \)), and integrability (cut-off assumption)  
		\begin{equation}\label{cutoff b3}
			\|b_3\|:= \sup_{\bm{\nu} \in \mathbb{E}_1} \int_{\mathbb{S}^{5}} b_3(\bm{\nu} \cdot \bm{\omega}, \omega_1 \cdot \omega_2) \, d\bm{\omega} < \infty.
		\end{equation}   
		Conservation of the symmetric velocity norm \( |\tilde{\bm{u}}| \) under ternary collisions, combined with the micro-reversibility conditions, ensures the invariance of the collision kernel for both central and adjacent interactions:  
		\begin{equation}\label{invariance of the collision kernel}
			B_3(\bm{u}^*, \bm{\omega}) = B_3(\bm{u}, \bm{\omega}), \quad B_3(\bm{u}_1^{1*}, \bm{\omega}) = B_3(\bm{u}_1, \bm{\omega}), 
		\end{equation} 
		where \( \bm{u}^* \) and \( \bm{u}_1^{1*} \) denote post-collisional relative velocities. 
		Moreover, the ternary collision operator introduced in \cite{MR4315662} corresponds to the hard sphere model under the generalized framework \eqref{B_3} with the specific parameter values (see Remark 2.3 in \cite{MR4412071}):  
		\[
		\gamma_3 = 1, \quad b_3(z, w) = \frac{1}{2} \frac{|z|}{\sqrt{1 + w}}.
		\]

		Based on the properties of the symmetrized ternary Boltzmann operator \(Q_T(f,f,f)\) presented in Appendix \ref{A.1.}, together with the known properties of the binary collision operator \(Q_B(f,f)\) established in \cite{MR1307620, Glassey1996}, we can deduce the following key properties of equation \eqref{BT-f}. The integrated conservation and entropy identities below are understood for sufficiently regular distributions for which the displayed integrals are finite and the spatial boundary terms vanish. In the whole-space perturbative setting around a global Maxwellian, we use the local conservation laws instead.
		\begin{enumerate}
			\item[$\bullet$] Conservation Laws.
			
			The conservation laws of total mass, momentum, and energy may be deduced from the identity  
			\[
			\int_{\mathbb{R}^{3}} \left[ Q_{B}(f,f) + Q_{T}(f,f,f) \right] \begin{pmatrix} 1 \\ v \\ |v|^{2} \end{pmatrix}  dv = 0.
			\]  
			By integrating equation \eqref{BT-f} over velocity and physical space, it follows that  
			\[
			\frac{d}{dt} \int_{\mathbb{R}^{3}} \int_{\mathbb{R}^{3}} f(t,x,v) \begin{pmatrix} 1 \\ v \\ |v|^{2} \end{pmatrix}  dv  dx = 0.
			\]
			
			\item[$\bullet$] H-theorem.
			
			The H-theorem for the binary-ternary Boltzmann equation follows from the entropy production functional  
			\[
			D(f) := \int_{\mathbb{R}^3} \left[ Q_{B}(f,f) + Q_{T}(f,f,f) \right] \ln f  \, dv \leq 0,
			\]  
			which implies that  
			\[
			\frac{d}{dt} \int_{\mathbb{R}^{3}} \int_{\mathbb{R}^{3}} f \ln f  \, dv \, dx = \int_{\mathbb{R}^3}D(f)\, dx \le 0.
			\]
			
			\item[$\bullet$] Equilibrium State.
			For positive distributions with sufficient velocity decay,
			
			\begin{align*}
				&\quad Q_{B}(f,f) + Q_{T}(f,f,f) = 0, \\ \iff& \quad   D(f) := \int_{\mathbb{R}^3} \left[ Q_{B}(f,f) + Q_{T}(f,f,f) \right] \ln f  \, dv = 0, \\
				\iff& \quad  f \text{ is a Maxwellian of the form:} \\
				& \qquad f(t, x, v) = \frac{R(t, x)}{(2\pi T(t, x))^{3/2}} e^{-\frac{|v - U(t, x)|^2}{2T(t, x)}}.
			\end{align*} 
		\end{enumerate}
	
		 Indeed, for the binary-ternary Boltzmann equation \eqref{BT-f}, we are concerned with its global well-posedness. The space inhomogeneous binary-ternary Boltzmann equation was studied by Ampatzoglou–Gamba–Pavlović–Tasković in \cite{MR4412071}, where they established global well-posedness near vacuum in Maxwellian-weighted \( L^\infty \)-spaces. A notable refinement was provided by Wei–Yu in \cite{JinboB}, who extended this result to an integrable space under the near-vacuum assumption. For the space homogeneous case, Ampatzoglou–Gamba–Pavlović–Tasković demonstrated in \cite{ampatzoglou2022momentestimateswellposednessbinaryternary} the generation and propagation of polynomial and exponential moments, along with global well-posedness. 
		While these results address well-posedness in specific regimes, the problem near global equilibrium remains open.
		In this paper, we investigate the global well-posedness of the binary-ternary Boltzmann equation \eqref{BT-f} in precisely this setting. Specifically, we define the normalized global equilibrium \( \mu(v) \) as
		\[
		\mu(v) := \frac{1}{(2\pi)^{\frac{3}{2}}} e^{-\frac{|v|^2}{2}}.
		\]
		The objective of this study is to establish that a particle distribution function \( f \) governed by the binary-ternary Boltzmann equation \eqref{BT-f}, evolving from an initial perturbation \( f_0 \) sufficiently close to \( \mu(v) \), admits a global solution. 
		To formulate this, we define the perturbation \( g(t, x, v) \) relative to \( \mu(v) \) via the decomposition:  
		\[
		f(t, x, v) = \mu(v) + \sqrt{\mu(v)} \, g(t, x, v).  
		\]  
		Substituting this ansatz into \eqref{BT-f} and dividing through by \( \sqrt{\mu(v)} \), we derive the Cauchy problem for the perturbation \( g(t, x, v) \):  
		\begin{equation}\label{BT-g}
			\begin{cases}  
				\partial_t g + v \cdot \nabla_x g + L g = \Gamma_{B}(g, g) + \Gamma_{T}(g, g) + T(g, g, g), \\  
				g(0, x, v) = g_0(x, v).
			\end{cases} 
		\end{equation} 
		The linearized collision operator \( L \) is decomposed as 
		\[L = L_{B} + L_{T},\] 
		where \( L_{B} \) and \( L_{T} \) are the linearized operators derived from the binary collision operator \( Q_{B} \) and ternary collision operator \( Q_{T} \), respectively. 
		Following classical results for the binary Boltzmann equation (see \cite{Glassey1996}), \( L_{B} \) admits the structure:  
		\[
		L_{B} = \nu_{B} - K_{B},
		\]  
		where $ \nu_{B}(v) $  represents the binary collision frequency. 
		The ternary counterpart \( L_{T} \), derived analogously from \( Q_{T} \), will inherit similar structural properties but with contributions from three-body interactions, as detailed in Section \ref{sec;Preliminaries}.
		
		The approach employed in this work to construct solutions to the binary-ternary Boltzmann equation \eqref{BT-f} is based on the energy method within an \(L^2\) framework—a robust technique for establishing global well-posedness near equilibrium for the Boltzmann equation and its variants.  
		This methodology was pioneered by Liu–Yang–Yu \cite{MR2043729} and Guo \cite{MR1908664,guo2003vlasov,MR2095473,MR2013332} for the classical Boltzmann equation and its variants near equilibrium.  
		Moreover, this approach has been successfully applied to a wide range of related models—including the non-cutoff Boltzmann equation as studied by Gressman–Strain \cite{MR2784329} and Alexandre–Morimoto–Ukai–Xu–Yang \cite{MR2793203,MR2795331,MR2863853}; the Boltzmann equation with external forces, such as the Vlasov–Poisson–Boltzmann system (Yang–Yu–Zhao \cite{MR2276498}, Duan–Yang–Zhao \cite{MR2911838,MR3037299}, Duan–Liu \cite{MR3116314}, and Xiao–Xiong–Zhao \cite{MR3567504}) and the Vlasov–Maxwell–Boltzmann system (Duan–Lei–Yang–Zhao \cite{duan2017vlasov}); the quantum Boltzmann equation by Ouyang–Wu \cite{MR4377165} and Zhou \cite{MR4623348}; the relativistic Boltzmann equation by Jang–Strain \cite{MR4470411}; and the relativistic quantum Boltzmann equation by Bae–Jang–Yun \cite{MR4264953}.
		These diverse applications further attest to the versatility and robustness of the method.
		
		To present the results in this paper, the following notations are introduced.
		For multi-indices $ \alpha=(\alpha_{1},\alpha_{2},\alpha_{3})\in\mathbb{N}^{3} ,$ we define 
		$$ \partial^{\alpha} =\partial^{\alpha_{1}}_{x_{1}}\partial^{\alpha_{2}}_{x_{2}}\partial^{\alpha_{3}}_{x_{3}}. $$
		The length of $ \alpha $ is $|\alpha|=\alpha_{1}+\alpha_{2}+\alpha_{3}.$
		Multi-indices are added according to the rule that if $ \alpha'=(\alpha_{1}',\alpha_{2}',\alpha_{3}') $ and $ \alpha''=(\alpha_{1}'',\alpha_{2}'',\alpha_{3}'') ,$ then $ \alpha'+\alpha''=(\alpha_{1}'+\alpha_{1}'',\alpha_{2}'+\alpha_{2}'',\alpha_{3}'+\alpha_{3}'') .$
		In addition, \( \langle\cdot, \cdot\rangle \) denotes the standard \( L^2 \) inner product in \( \mathbb{R}^3_v \), and \( (\cdot, \cdot) \) represents the inner product in \( \mathbb{R}^3_x \times \mathbb{R}^3_v \).  
		For the sake of convenience in the notations frequently used throughout this paper,  we denote the dissipation norm as
		\[
		 \|\cdot\|_{L^{2}_{\mathcal{D}}} := \|\sqrt{\nu_{B}(v)}(\cdot)\|_{L^{2}_{v}}.
		\]
		For any \( q \in \mathbb{R} \), we define the weighted \( L^2 \)-norm \( \|\cdot\|_{L^2_q} \) by  
		\[
		\|\cdot\|_{L^2_q} := \left\| \langle v \rangle^{\frac{q}{2}} (\cdot) \right\|_{L^2_v},  
		\]  
		where \( \langle \cdot \rangle := \sqrt{1 + |\cdot|^2} \) denotes the Japanese bracket.
		We also define 
		$$ \|\partial^{\alpha}[u_{1},\dots,u_{n}]\|^{2}_{L^{2}_{x}}:=\sum_{k=1}^{n}\|\partial^{\alpha} u_{k} \|^{2}_{L^{2}_{x}}.$$
		Moreover, the notation \( A \lesssim B \) indicates that there exists a universal constant \( C > 0 \) such that \( A \leq C B \). 
		The notation \( A \sim B \) means that there exist two positive constants \( C_1 \) and \( C_2 \) such that \( C_1 B \leq A \leq C_2 B \).

     The core idea of this energy method is as follows. The approach begins with the selection of an energy functional, which is typically defined within (weighted) high-order Sobolev spaces grounded in an \(L^2\) framework. In this paper, we specifically choose the functional \(\|\cdot\|^{2}_{H^N_x L^2_v}\).  
     Furthermore, a key step involves identifying a dissipation rate that bounds the linear term from below and controls the nonlinear terms from above. An important observation is that the linearized collision operator \(L\) has a five-dimensional kernel spanned by the collision invariants (see Proposition \ref{prop:L_properties}):
     \[
     \operatorname{Ker}(L) = \operatorname{span} \left\{ \sqrt{\mu},  v_1\sqrt{\mu},  v_2\sqrt{\mu},  v_3\sqrt{\mu},  |v|^2\sqrt{\mu} \right\}.
     \]
     This structure motivates the macro-micro decomposition:
     \[
     g =  \mathbf{P}g + (\mathbf{I} - \mathbf{P})g,
     \]
     where \(\mathbf{P}\) is the orthogonal projection onto \(\operatorname{Ker}(L)\). Due to the coercivity of \(L\), a lower bound is only available for the microscopic component \((\mathbf{I} - \mathbf{P})g\):
     \[
     \langle L g, g \rangle \gtrsim \|(\mathbf{I} - \mathbf{P})g\|_{L^2_{\mathcal{D}}}^2,
     \]
     implying that the dissipation directly derived from \(L\) is only microscopic. 
     To obtain a complete dissipation mechanism, the macroscopic equations—which behave like an elliptic system—are used to estimate the macroscopic components (see Subsection 4.2). In this paper, we introduce a dissipation rate of the form:
     \[\mathcal{D}_{N}(g(t))\sim \|\nabla_x \mathbf{P}g(t)\|^{2}_{H^{N-1}_{x}L^{2}_{\mathcal{D}}}+ \|(\mathbf{I} - \mathbf{P})g(t)\|^{2}_{H^{N}_{x}L^{2}_{\mathcal{D}}},\]
    with the precise definition given in \eqref{dissipation rate}.
		
		Building upon this framework, we now present the main result of this paper: the global existence and uniqueness of solutions to the Cauchy problem \eqref{BT-g}.  
\begin{theorem}\label{main_theorem}  
			Suppose the initial data for \eqref{BT-g} satisfies  
			$f_0(x,v) = \mu(v) + \sqrt{\mu(v)} \, g_0(x,v)\ge0. $ 
		Let \( N\ge2 \), \( \gamma_2 \in (-\frac{3}{2}, 1] \), \( \gamma_3 \in (-3, 1] \), and assume that \( \gamma_3 \leq \gamma_2 \). Then, there exists a constant \( \epsilon > 0 \) such that if the initial perturbation satisfies  
		\[
		\|g_0\|^{2}_{H^N_x L^2_v} \leq \epsilon,  
		\]  
		the binary-ternary Boltzmann equation \eqref{BT-g} admits a unique global-in-time solution \( g(t,x,v) \) in \((0,+\infty)\times\mathbb{R}^{6}_{x,v}\) satisfying \( f(t,x,v) = \mu(v) + \sqrt{\mu(v)} \, g(t,x,v) \geq 0 \) and  
		\[
		\sup_{t\in\mathbb{R}^{+}}\|g(t)\|^{2}_{H^N_x L^2_v}  + \int_{0}^{\infty}\mathcal{D}_{N}(g)(t)\,dt \lesssim \|g_0\|^{2}_{H^N_x L^2_v}.  
		\]
		\end{theorem}  
	  
	  \begin{remark}
	  	The parameters \(\gamma_2\) and \(\gamma_3\) in our main result cover both hard potentials and a portion of soft potentials. In particular, the hard sphere model is included by taking $\gamma_2=\gamma_3=1$ and choosing the collision kernels as follows:  
	  	\[
	  	B_2(u, \omega) = \frac{1}{2}|u \cdot \omega|, 
	  	\]  \[
	  	B_3(\bm{u}, \bm{\omega}) = \frac{1}{2} \frac{|\bm{u} \cdot \bm{\omega}|}{\sqrt{1 + \omega_1 \cdot \omega_2}}, \quad
	  	B_3(\bm{u}_{1}, \bm{\omega}) = \frac{1}{2} \frac{|\bm{u}_{1} \cdot \bm{\omega}|}{\sqrt{1 + \omega_1 \cdot \omega_2}}.
	  	\]  
	  	Moreover, the requirement that \(\gamma_3 \leq \gamma_2\) is due to a technical reason: in the definition of the dissipation rate, we only make use of the coercivity of the binary linearized collision operator \(L_B\).
	  \end{remark}
	
	This paper is structured as follows. 
	The reformulated equation \eqref{BT-g} is rigorously derived in Section \ref{sec;Preliminaries} by linearizing \eqref{BT-f} about the global equilibrium. Furthermore, key properties of the linearized collision operator \( L \) are established, and the relevant notations are explained in detail. 
	Section \ref{sect.Estimates of the Nonlinear Terms} presents several essential estimates for both the linear and nonlinear terms within the energy method framework. 
	Section \ref{sec;Global Existence} is devoted to proving the local well-posedness of the solution, which is then extended globally by combining energy estimates with macroscopic dissipation. 
	In Appendix \ref{sec;app}, we review several key properties of the ternary Boltzmann collision operator \(Q_T\), as established in Section 5.4 of \cite{MR4337060}. Additionally, we provide a summary of convolution estimates and Sobolev inequalities that are frequently used throughout this paper.

	    \section{Preliminaries}\label{sec;Preliminaries} 
	    
	    In this section, we rigorously derive the reformulated equation \eqref{BT-g} by  linearizing  the binary-ternary Boltzmann equation \eqref{BT-f} about the global equilibrium \( \mu(v) \)  via the perturbation  \( f = \mu + \sqrt{\mu} \, g \).
	    A detailed explanation of the notations is provided, and key properties of the linearized collision operator \( L \) are established.
		
		The  gain and loss terms of  binary collision operators are defined as follows:
		\begin{align}
			Q_{B,+}(g, h) &= \int_{\mathbb{S}^{2} \times \mathbb{R}^3} B_{2}(u, \omega) \, g' h_1' \, d\omega dv_1, \label{eq:Qz1} \\
			Q_{B,-}(g, h) &= \int_{\mathbb{S}^{2} \times \mathbb{R}^3} B_{2}(u, \omega) \, g h_1 \, d\omega dv_1 \nonumber \\
			&= g \int_{\mathbb{S}^{2} \times \mathbb{R}^3} B_{2}(u, \omega) \, h_1 \, d\omega dv_1 \nonumber \\
			&=: g R_{B}(h). \label{eq:Qz2}
		\end{align} 
		Furthermore, by leveraging the conservation of energy in binary collisions \eqref{b conservation}, we observe that
		\begin{equation}\label{b conservation mu}
			\mu\mu_{1}=\mu'\mu_1',
		\end{equation}
		so we define
		\begin{align} 
			\Gamma_{B,+}(g, h) :=\frac{1}{\sqrt{\mu}} Q_{B,+}(\sqrt{\mu}g, \sqrt{\mu}h)&= \int_{\mathbb{S}^{2} \times \mathbb{R}^3} \sqrt{\mu_{1}}B_{2}(u, \omega) \, g' h_1' \, d\omega dv_1, \label{eq:Gamma_plus} \\
			\Gamma_{B,-}(g, h) :=\frac{1}{\sqrt{\mu}} Q_{B,-}(\sqrt{\mu}g, \sqrt{\mu}h)&= \int_{\mathbb{S}^{2} \times \mathbb{R}^3} \sqrt{\mu_{1}}B_{2}(u, \omega) \, g h_1 \, d\omega dv_1, \label{eq:Gamma_minus} \\
			\Gamma_{B}(g, h) :=\frac{1}{\sqrt{\mu}} Q_{B}(\sqrt{\mu}g, \sqrt{\mu}h)&= \Gamma_{B,+}(g, h) - \Gamma_{B,-}(g, h). \label{eq:Gamma_total}
		\end{align}
		The binary linearized operator \( L_B \) is expressed as:
		\begin{align} 
			L_{B}g &= -\Gamma_{B}(\sqrt{\mu}, g) - \Gamma_{B}(g, \sqrt{\mu}) = \nu_{B}(v) g - K_{B}g, \label{eq:LB_decomp}
		\end{align}
		where the binary collision frequency \( \nu_B(v) \) is defined by:
		\begin{equation*} 
			\nu_{B}(v) = \int_{\mathbb{R}^3 \times \mathbb{S}^2} B_2(u, \omega) \mu_1 \, d\omega dv_1.
		\end{equation*} 
	It follows easily from \eqref{cutoff b2} and Lemma \ref{Convolution estimates} that for $\gamma_{2}\in(-3,1],$
	 \begin{equation}\label{nuB}
	 	\nu_{B}(v) =b\int_{\mathbb{R}^3} |v_{1}-v|^{\gamma_{2}} \mu_1 dv_1\sim \langle v\rangle^{\gamma_{2}}.
	 \end{equation} 
		The operator \( K_B \) is compact for hard-sphere interactions in \( L^2 \) space (see \cite{Glassey1996}, Section 3.5), and ``almost compact" for soft potentials in weighted spaces \( L^2_\mathcal{D} \) (see \cite{MR2013332}, Lemma 2).
		  
		Now we define the gain and loss terms of the ternary collision operator \( Q_T \):
		\begin{align}
			Q_{T,+}(g, h, \xi) &:= \int_{\mathbb{S}^{5} \times \mathbb{R}^{6}}B_3(\bm{u}, \bm{\omega})g^* h_1^* \xi_2^* + 2B_3(\bm{u}_1, \bm{\omega})  g^{1*} h_1^{1*} \xi_2^{1*} d\bm{\omega} dv_{1}dv_{2}, \label{eq:QT_plus} \\
			Q_{T,-}(g, h, \xi) &:= \int_{\mathbb{S}^{5} \times \mathbb{R}^{6}} \left( B_3(\bm{u}, \bm{\omega})+2B_3(\bm{u}_1, \bm{\omega}) \right)g h_1 \xi_2 d\bm{\omega} dv_{1}dv_{2} \nonumber \\
			&= g R_T(h, \xi), \label{eq:QT_minus}
		\end{align}
		where the  operator \( R_T \) is defined as:
		\[
		R_T(h, \xi) := \int_{\mathbb{S}^{5} \times \mathbb{R}^{6}} \left( B_3(\bm{u}, \bm{\omega})+2B_3(\bm{u}_1, \bm{\omega}) \right) h_1 \xi_2 d\bm{\omega} dv_{1}dv_{2} .
		\]
		Furthermore, due to the conservation of energy in ternary collisions as described in  \eqref{t conservation}, we can see that
		\begin{equation}\label{t conservation mu}
			\mu\mu_{1}\mu_{2}=\mu^{*}\mu^{*}_{1}\mu^{*}_{2}=\mu^{1*}\mu^{1*}_{1}\mu^{1*}_{2},
		\end{equation}
		which leads us to define the trilinear operators \( T_\pm \) and their combination: 
		\small
		\begin{align}
			T_+(g, h, \xi) &:= \frac{1}{\sqrt{\mu}} Q_{T,+}(\sqrt{\mu} g, \sqrt{\mu} h, \sqrt{\mu} \xi) \nonumber \\
			&= \int_{\mathbb{S}^{5} \times \mathbb{R}^{6}}\sqrt{\mu_{1}\mu_2}  \left(B_3(\bm{u}, \bm{\omega})g^* h_1^* \xi_2^* + 2B_3(\bm{u}_1, \bm{\omega})  g^{1*} h_1^{1*} \xi_2^{1*}\right) d\bm{\omega} dv_{1}dv_{2},  \label{eq:T_plus} \\
			T_-(g, h, \xi) &:= \frac{1}{\sqrt{\mu}} Q_{T,-}(\sqrt{\mu} g, \sqrt{\mu} h, \sqrt{\mu} \xi) \nonumber \\
			&= \int_{\mathbb{S}^{5} \times \mathbb{R}^{6}} \sqrt{\mu_{1}\mu_2} \left( B_3(\bm{u}, \bm{\omega})+2B_3(\bm{u}_1, \bm{\omega}) \right)g h_1 \xi_2 d\bm{\omega} dv_{1}dv_{2}, \label{eq:T_minus}\\
			T(g, h, \xi) &:= T_+(g, h, \xi)-T_-(g, h, \xi).\label{eq:T}
		\end{align}
	\normalsize
		The ternary linearized operator \(L_T \) and the ternary collision frequency \( \nu_T \) are: 
		\begin{equation}\label{L_T}
			L_T g = -T(g, \sqrt{\mu}, \sqrt{\mu}) - T(\sqrt{\mu}, g, \sqrt{\mu}) - T(\sqrt{\mu}, \sqrt{\mu},g) = \nu_T(v) g - K_T g,
		\end{equation} 
		where \( \nu_T(v) \) is defined as:
		\[
		\nu_T(v) =\int_{\mathbb{S}^{5} \times \mathbb{R}^{6}}\left( B_3(\bm{u}, \bm{\omega})+2B_3(\bm{u}_1, \bm{\omega}) \right) {\mu_{1}\mu_2}  d\bm{\omega} dv_{1}dv_{2}.
		\] 
		According to \eqref{tildeu-u}, \eqref{cutoff b3}, and Lemma \ref{Convolution estimates}, we observe that for $\gamma_{3}\in(-6,1],$
		\begin{equation}\label{nu3}
			0\le\nu_T(v)\lesssim\int_{\mathbb{R}^{6}}|{\bm{u}}|^{\gamma_3} {\mu_{1}\mu_2} dv_{1}dv_{2}\lesssim\langle v\rangle^{\gamma_{3}}.
		\end{equation}
		
		By substituting the perturbative decomposition \( f = \mu + \sqrt{\mu} \, g \) into the binary-ternary Boltzmann equation \eqref{BT-f}, we derive the linearized equation for the perturbation \( g(t, x, v) \) as described in \eqref{BT-g}:  
		\[
		\begin{cases}  
			\partial_t g + v \cdot \nabla_x g + L g = \Gamma_{{B}}(g, g) + \Gamma_{{T}}(g, g) + T(g, g, g), \\  
			g(0, x, v) = g_0(x, v),  
		\end{cases}   
		\]  
		where \( g_0(x, v) = ({f_0(x, v) - \mu(v)})/{\sqrt{\mu(v)}} \) defines the initial perturbation. 
		The linearized collision operator \( L \) is decomposed as  
		\[
		L = L_{{B}} + L_{{T}},  
		\]  
		with \( L_{{B}} \) and \( L_{{T}} \) corresponding to the linearization of the binary and ternary collision operators, respectively (defined explicitly in  \eqref{eq:LB_decomp} and \eqref{L_T}). The nonlinear terms \( \Gamma_{{B}}(g, h) \) and \( T(g, h, \xi) \) are given by \eqref{eq:Gamma_total} and \eqref{eq:T}, while the ternary interaction term \( \Gamma_{{T}}(g, h) \) is expressed as  
		\begin{equation}\label{Gamma_{{T}}(g,h)}
			\Gamma_{{T}}(g,h) := T(g, h, \sqrt{\mu}) + T(g, \sqrt{\mu}, h) + T(\sqrt{\mu}, g, h).  
		\end{equation} 
		Furthermore, the conservation laws stated in Proposition \ref{prop:conservation} (mass, momentum, and energy) can be reformulated in terms of the perturbation \( g \):  
		\begin{align}\label{conservation laws g}
			\begin{aligned}  
				&\partial_t \int_{\mathbb{R}^3} \sqrt{\mu} \, g \, dv + \nabla_x \cdot \int_{\mathbb{R}^3} v \sqrt{\mu} \, g \, dv = 0, \\  
				&\partial_t \int_{\mathbb{R}^3} v \sqrt{\mu} \, g \, dv + \nabla_x \cdot \int_{\mathbb{R}^3} v \otimes v \sqrt{\mu} \, g \, dv = 0, \\  
				&\partial_t \int_{\mathbb{R}^3} |v|^2 \sqrt{\mu} \, g \, dv + \nabla_x \cdot \int_{\mathbb{R}^3} |v|^2 v \sqrt{\mu} \, g \, dv = 0.  
			\end{aligned} 
		\end{align} 
		
		The following two propositions will establish key properties of the linearized collision operators \( L_T \) and \( L \), which play crucial roles in the construction of solutions to equation \eqref{BT-g}.
		
		\begin{proposition}[Properties of the Linearized Ternary Operator]\label{prop:LT_properties}
			The linearized ternary operator \( L_T \) satisfies the following:
			\begin{enumerate}
				\item[(\romannumeral 1)] \( L_T \) is a symmetric operator on \( L^2_v \):  
				\[
				\langle L_T g, h \rangle = \langle g, L_T h \rangle.
				\]
				\item[(\romannumeral 2)] \( L_T \) is a positive operator on \( L^2_v \):
				\[
				\langle L_T g, g \rangle \geq 0,
				\]
				with equality holding if and only if
				\[
				g = \sqrt{\mu}(a + \mathbf{b} \cdot v + c|v|^2)
				\]
				for some \( a, c \in \mathbb{R} \) and \( \mathbf{b} \in \mathbb{R}^3 \).
				\item[(\romannumeral 3)] \[
				\operatorname{Ker}(L_T) = \operatorname{span} \left\{ \sqrt{\mu}, \, v_1\sqrt{\mu}, \, v_2\sqrt{\mu}, \, v_3\sqrt{\mu}, \, |v|^2\sqrt{\mu} \right\}.
				\]
			\end{enumerate}    
		\end{proposition}
		\begin{proof}
			From the definition of \( L_{T} \), we derive that
			\begin{align*}
				L_{T}g=  &\int_{\mathbb{S}^{5} \times \mathbb{R}^{6}}\sqrt{\mu}\mu_{1}\mu_{2}\bigg(B_3(\bm{u}, \bm{\omega}) \big[(\mu^{-\frac{1}{2}}g-(\mu^{-\frac{1}{2}}g)^{*})+((\mu^{-\frac{1}{2}}g)_{1}-(\mu^{-\frac{1}{2}}g)^{*}_{1}) \\
				&\quad+((\mu^{-\frac{1}{2}}g)_{2}-(\mu^{-\frac{1}{2}}g)^{*}_{2})\big]+ 2B_3(\bm{u}_1, \bm{\omega}) \big[(\mu^{-\frac{1}{2}}g-(\mu^{-\frac{1}{2}}g)^{1*})\\[0.5mm]
				&\quad +((\mu^{-\frac{1}{2}}g)_{1}-(\mu^{-\frac{1}{2}}g)^{1*}_{1})+((\mu^{-\frac{1}{2}}g)_{2}-(\mu^{-\frac{1}{2}}g)^{1*}_{2})\big] \bigg)   d\bm{\omega} dv_{1}dv_{2},
			\end{align*}  
			from which we conclude that 
			\small
			\begin{align*}
				\langle L_T g, h \rangle=  &\int_{\mathbb{S}^{5} \times \mathbb{R}^{9}} {\mu}\mu_{1}\mu_{2} B_3(\bm{u}, \bm{\omega}) \bigg[\mu^{-\frac{1}{2}}g+(\mu^{-\frac{1}{2}}g)_{1}+(\mu^{-\frac{1}{2}}g)_{2} \\
				&\quad\quad-(\mu^{-\frac{1}{2}}g)^{*} -(\mu^{-\frac{1}{2}}g)^{*}_{1}-(\mu^{-\frac{1}{2}}g)^{*}_{2}\bigg]\mu^{-\frac{1}{2}}h  d\bm{\omega} dv_{1}dv_{2}dv\\[0.5mm] 
			    &+2\int_{\mathbb{S}^{5} \times \mathbb{R}^{9}} {\mu}\mu_{1}\mu_{2} B_3(\bm{u}_1, \bm{\omega}) \bigg[\mu^{-\frac{1}{2}}g+(\mu^{-\frac{1}{2}}g)_{1}+(\mu^{-\frac{1}{2}}g)_{2} \\
			    &\quad\quad-(\mu^{-\frac{1}{2}}g)^{1*} -(\mu^{-\frac{1}{2}}g)^{1*}_{1}-(\mu^{-\frac{1}{2}}g)^{1*}_{2}\bigg]  \mu^{-\frac{1}{2}}h  d\bm{\omega} dv_{1}dv_{2}dv\\
				=:&I+II.
			\end{align*} 
			\normalsize
			By employing the variable transformation \( T_{1,\bm{\omega}}: (v, v_1, v_2) \mapsto (v^*, v_1^*, v_2^*) \) and noting that \( |\det T_{1,\bm{\omega}}| = 1 ,\) \eqref{t conservation mu} and invariance of the collision kernel \eqref{invariance of the collision kernel}, we deduce that 
			\small
			\begin{align*}
				I=  &\int_{\mathbb{S}^{5} \times \mathbb{R}^{9}} {\mu}\mu_{1}\mu_{2} B_3(\bm{u}, \bm{\omega}) \bigg[\mu^{-\frac{1}{2}}g+(\mu^{-\frac{1}{2}}g)_{1}+(\mu^{-\frac{1}{2}}g)_{2} \\
				&\quad\quad-(\mu^{-\frac{1}{2}}g)^{*} -(\mu^{-\frac{1}{2}}g)^{*}_{1}-(\mu^{-\frac{1}{2}}g)^{*}_{2}\bigg]\mu^{-\frac{1}{2}}h  d\bm{\omega} dv_{1}dv_{2}dv\\[0.5mm] 
				=  &\frac{1}{2}\int_{\mathbb{S}^{5} \times \mathbb{R}^{9}} {\mu}\mu_{1}\mu_{2} B_3(\bm{u}, \bm{\omega}) \bigg[\mu^{-\frac{1}{2}}g+(\mu^{-\frac{1}{2}}g)_{1}+(\mu^{-\frac{1}{2}}g)_{2} \\
				&\quad\quad-(\mu^{-\frac{1}{2}}g)^{*} -(\mu^{-\frac{1}{2}}g)^{*}_{1}-(\mu^{-\frac{1}{2}}g)^{*}_{2}\bigg]\bigg(\mu^{-\frac{1}{2}}h -(\mu^{-\frac{1}{2}}h)^{*}\bigg) d\bm{\omega} dv_{1}dv_{2}dv .
			\end{align*} 
			\normalsize
			Additionally, following the proof of Proposition 5.5. in \cite{MR4337060}  and using the variable transformations \( (v, v_1 ) \mapsto (v_1, v) \) and \( (\omega_{1}, \omega_{2}, v_1, v_2) \mapsto (\omega_{2}, \omega_{1}, v_2, v_1) \), we arrive at the conclusion that 
			\small
			\begin{align*}
				II=  &2\int_{\mathbb{S}^{5} \times \mathbb{R}^{9}} {\mu}\mu_{1}\mu_{2} B_3(\bm{u}_1, \bm{\omega}) \bigg[\mu^{-\frac{1}{2}}g+(\mu^{-\frac{1}{2}}g)_{1}+(\mu^{-\frac{1}{2}}g)_{2} \\
				&\quad\quad-(\mu^{-\frac{1}{2}}g)^{1*} -(\mu^{-\frac{1}{2}}g)^{1*}_{1}-(\mu^{-\frac{1}{2}}g)^{1*}_{2}\bigg]  \mu^{-\frac{1}{2}}h  d\bm{\omega} dv_{1}dv_{2}dv\\
				=&2\int_{\mathbb{S}^{5} \times \mathbb{R}^{9}} {\mu}\mu_{1}\mu_{2} B_3(\bm{u}, \bm{\omega}) \bigg[\mu^{-\frac{1}{2}}g+(\mu^{-\frac{1}{2}}g)_{1}+(\mu^{-\frac{1}{2}}g)_{2} \\
				&\quad\quad-(\mu^{-\frac{1}{2}}g)^{*} -(\mu^{-\frac{1}{2}}g)^{*}_{1}-(\mu^{-\frac{1}{2}}g)^{*}_{2}\bigg] (\mu^{-\frac{1}{2}}h )_{1} d\bm{\omega} dv_{1}dv_{2}dv\\
				=&\int_{\mathbb{S}^{5} \times \mathbb{R}^{9}} {\mu}\mu_{1}\mu_{2} B_3(\bm{u}, \bm{\omega}) \bigg[\mu^{-\frac{1}{2}}g+(\mu^{-\frac{1}{2}}g)_{1}+(\mu^{-\frac{1}{2}}g)_{2} \\
				&\quad\quad-(\mu^{-\frac{1}{2}}g)^{*} -(\mu^{-\frac{1}{2}}g)^{*}_{1}-(\mu^{-\frac{1}{2}}g)^{*}_{2}\bigg] \bigg((\mu^{-\frac{1}{2}}h )_{1}+(\mu^{-\frac{1}{2}}h )_{2}\bigg) d\bm{\omega} dv_{1}dv_{2}dv,
			\end{align*} 
			\normalsize
			which implies that 
			\small
			\begin{align*}
				II =&\frac{1}{2}\int_{\mathbb{S}^{5} \times \mathbb{R}^{9}} {\mu}\mu_{1}\mu_{2} B_3(\bm{u}, \bm{\omega}) \bigg[\mu^{-\frac{1}{2}}g+(\mu^{-\frac{1}{2}}g)_{1}+(\mu^{-\frac{1}{2}}g)_{2}-(\mu^{-\frac{1}{2}}g)^{*} -(\mu^{-\frac{1}{2}}g)^{*}_{1} \\
				&\quad\quad-(\mu^{-\frac{1}{2}}g)^{*}_{2}\bigg]   \bigg((\mu^{-\frac{1}{2}}h )_{1}+(\mu^{-\frac{1}{2}}h )_{2}-(\mu^{-\frac{1}{2}}h)^{*}_{1}-(\mu^{-\frac{1}{2}}h)^{*}_{2}\bigg) d\bm{\omega} dv_{1}dv_{2}dv,
			\end{align*} 
			\normalsize
			by applying the variable transformation \( T_{1,\bm{\omega}}: (v, v_1, v_2) \mapsto (v^*, v_1^*, v_2^*) \) again.
			Hence, we conclude that
			\small
			\begin{align*}
				\langle L_T g, h \rangle=&I+II\\
				= &\frac{1}{2}\int_{\mathbb{S}^{5} \times \mathbb{R}^{9}} {\mu}\mu_{1}\mu_{2} B_3(\bm{u}, \bm{\omega}) \bigg[\mu^{-\frac{1}{2}}g+(\mu^{-\frac{1}{2}}g)_{1}+(\mu^{-\frac{1}{2}}g)_{2}  \\
				&\quad\quad-(\mu^{-\frac{1}{2}}g)^{*}-(\mu^{-\frac{1}{2}}g)^{*}_{1}-(\mu^{-\frac{1}{2}}g)^{*}_{2}\bigg]   \bigg[\mu^{-\frac{1}{2}}h+(\mu^{-\frac{1}{2}}h)_{1}+(\mu^{-\frac{1}{2}}h)_{2}  \\
				&\quad\quad-(\mu^{-\frac{1}{2}}h)^{*}-(\mu^{-\frac{1}{2}}h)^{*}_{1}-(\mu^{-\frac{1}{2}}h)^{*}_{2}\bigg]  d\bm{\omega} dv_{1}dv_{2}dv\\
				=&\langle g,  L_Th \rangle,
			\end{align*} 
			\normalsize
			and
			\small
			\begin{align*}
				\langle L_T g, g \rangle
				= &\frac{1}{2}\int_{\mathbb{S}^{5} \times \mathbb{R}^{9}} {\mu}\mu_{1}\mu_{2} B_3(\bm{u}, \bm{\omega}) \bigg[\mu^{-\frac{1}{2}}g+(\mu^{-\frac{1}{2}}g)_{1}+(\mu^{-\frac{1}{2}}g)_{2}  \\
				&\quad\quad-(\mu^{-\frac{1}{2}}g)^{*}-(\mu^{-\frac{1}{2}}g)^{*}_{1}-(\mu^{-\frac{1}{2}}g)^{*}_{2}\bigg]^{2}  d\bm{\omega} dv_{1}dv_{2}dv\\
				\ge&0.
			\end{align*} 
			\normalsize
			From the above equation, we can see that \(\langle L_T g, g \rangle = 0\) if and only if
			$$ \mu^{-\frac{1}{2}}g+(\mu^{-\frac{1}{2}}g)_{1}+(\mu^{-\frac{1}{2}}g)_{2} -(\mu^{-\frac{1}{2}}g)^{*}-(\mu^{-\frac{1}{2}}g)^{*}_{1}-(\mu^{-\frac{1}{2}}g)^{*}_{2}=0.$$
			By Proposition \ref{prop:collision_invariant}, this is also equivalent to
			$$g=\sqrt{\mu}(a +  \mathbf{b}\cdot v  + c|v|^2)$$
			for some \( a, c \in \mathbb{R} \) and \( \mathbf{b} \in \mathbb{R}^3 \). 
			Finally, it is evident from the principles of momentum and energy conservation \eqref{t conservation} that  \( g = \sqrt{\mu}(a + \mathbf{b} \cdot v + c|v|^2) \) satisfies \( L_T g = 0 \). 
			Consequently, the kernel of the linearized ternary operator is given by:  
			\[
			\operatorname{Ker}(L_T) = \operatorname{span} \left\{ \sqrt{\mu}, \, v_1\sqrt{\mu}, \, v_2\sqrt{\mu}, \, v_3\sqrt{\mu}, \, |v|^2\sqrt{\mu} \right\}.
			\]  
		\end{proof}
		
			\begin{proposition}[Properties of the Linearized  Operator]\label{prop:L_properties}
			The linearized operator \( L = L_{{B}} + L_{{T}} \) adheres to the following:
			\begin{enumerate}
				\item[(\romannumeral 1)] \( L \) is a symmetric operator on \( L^2_v \):  
				\[
				\langle Lg, h \rangle = \langle g, L h \rangle.
				\]
				\item[(\romannumeral 2)] \( L \) is a positive operator on \( L^2_v \):
				\[
				\langle L g, g \rangle \geq 0,
				\]
				with equality holding if and only if
				\[
				g = \sqrt{\mu}(a + \mathbf{b} \cdot v + c|v|^2)
				\]
				for some \( a, c \in \mathbb{R} \) and \( \mathbf{b} \in \mathbb{R}^3 \).
				\item[(\romannumeral 3)] \( Lg = 0 \) if and only if \( g = \mathbf{P}g \), where \( \mathbf{P} \) denotes the orthogonal projection onto the 5-dimensional subspace of collision invariants:
				\[
				\operatorname{span} \left\{ \sqrt{\mu}, \, v_1\sqrt{\mu}, \, v_2\sqrt{\mu}, \, v_3\sqrt{\mu}, \, |v|^2\sqrt{\mu} \right\}.
				\]
				This implies
				\[
				\operatorname{Ker}(L) = \operatorname{span} \left\{ \sqrt{\mu}, \, v_1\sqrt{\mu}, \, v_2\sqrt{\mu}, \, v_3\sqrt{\mu}, \, |v|^2\sqrt{\mu} \right\}.
				\]
				\item[(\romannumeral 4)] There exists a \( \delta_{0} > 0 \) such that  
				$$\langle L g, g \rangle \geq\delta_{0}\|\{\mathbf{I-P}\}g\|_{L^{2}_{\mathcal{D}}}^2. $$
			\end{enumerate}     
		\end{proposition}
		\begin{proof}
			As per the classical results, such as the analysis in Chapter 3 of \cite{Glassey1996}, Proposition \ref{prop:LT_properties} also holds for the binary linearized Boltzmann operator \( L_{{B}} \). Hence, in conjunction with Proposition \ref{prop:LT_properties},  (\romannumeral1), (\romannumeral2), and (\romannumeral3) of this proposition are evident.
			Furthermore, in light of \cite{Glassey1996,MR2013332}, for \( \gamma_{2} \in (-3,1] \), there exists a \( \delta_{0} > 0 \) such that
			$$
			\langle L_{B} g, g \rangle \geq \delta_{0} \| \{\mathbf{I-P}\} g \|_{L^{2}_{\mathcal{D}}}^2.
			$$
			Consequently, given that
			$$
			\langle L_{T} g, g \rangle \geq 0,
			$$
			it follows that there exists a \( \delta > 0 \) such that
			$$
			\langle L g, g \rangle = \langle L_{B} g, g \rangle + \langle L_{T} g, g \rangle \geq \delta_{0} \| \{\mathbf{I-P}\} g \|_{L^{2}_{\mathcal{D}}}^2.
			$$ 
		\end{proof}

		\section{Estimates of the Nonlinear Terms}\label{sect.Estimates of the Nonlinear Terms}

		In this section, we establish and prove several essential estimates required to control the nonlinear terms \(\Gamma_{B}\), \(\Gamma_{T}\), and \(T\) within the energy method framework. Furthermore, these estimates are also employed to derive corresponding bounds for the linear terms.

		\begin{lemma}\label{Estimates of GammaB}
			For \( \gamma_{2} \in (-\frac{3}{2},1] \), it is established that
			$$ |\langle\Gamma_{B,-}(g, h),\eta \rangle|\lesssim \|h\|_{L^{2}_{v}}\|g\|_{L^{2}_{\gamma_{2}}}\|\eta\|_{L^{2}_{\gamma_{2}}},$$
			and
			$$ |\langle\Gamma_{B,+}(g, h),\eta \rangle|\lesssim \big(\|g\|_{L^{2}_{v}}\|h\|_{L^{2}_{\gamma_{2}}}+\|h\|_{L^{2}_{v}}\|g\|_{L^{2}_{\gamma_{2}}}\big)\|\eta\|_{L^{2}_{\gamma_{2}}}.$$
		\end{lemma}
	    \begin{proof}
	    	The case for \( \gamma_{2} \in [0,1] \) can be established similarly by combining Lemma \ref{Convolution estimates} with the proof of Lemma 2.3 in \cite{MR1908664}. We therefore concentrate on the scenario where \( \gamma_{2} \in (-\frac{3}{2},0) \). 
	    	Starting with the estimate for the \( \Gamma_{B,-} \) term, we apply Lemma \ref{Convolution estimates} with \( d=3 \) and \( q=\gamma_{2} \), yielding
	    	\small
	    	\begin{align*}
	    		|\langle\Gamma_{B,-}(g, h),\eta \rangle| \leq&\int_{\mathbb{S}^{2} \times \mathbb{R}^6} \sqrt{\mu_{1}} |u|^{\gamma_{2}}\, b_2 \, |g h_1 \eta| \, d\omega dv_1dv  \\
	    		\lesssim&\int_{ \mathbb{R}^3} \bigg( \int_{\mathbb{S}^{2} \times \mathbb{R}^3} {\mu_{1}} |u|^{2\gamma_{2}}\, b_2d\omega dv_1\bigg)^{\frac{1}{2}} \bigg( \int_{\mathbb{S}^{2} \times \mathbb{R}^3} |h_1|^{2} b_2d\omega dv_1\bigg)^{\frac{1}{2}} |g \eta|  dv\\
	    		\lesssim& \|h\|_{L^{2}_{v}} \int_{ \mathbb{R}^3} \langle v\rangle^{\gamma_{2}} |g \eta|  dv\\
	    		\lesssim& \|h\|_{L^{2}_{v}}\|g\|_{L^{2}_{\gamma_{2}}}\|\eta\|_{L^{2}_{\gamma_{2}}}.
	    	\end{align*}
	    	\normalsize 
	    	Moving on to the more intricate \( \Gamma_{B,+} \) term, we first observe that
	    	$$ |\langle\Gamma_{B,+}(g, h),\eta \rangle| \le \int_{\mathbb{S}^{2} \times \mathbb{R}^6} \sqrt{\mu_{1}} |u|^{\gamma_{2}}\, b_2 \, |g' h_1' \eta| \, d\omega dv_1dv .$$ 
	    	We decompose the integration region into
	    	$$\left\{|v_1| \ge \frac{|v|}{2}\right\} \cup \left\{|v_1| \le \frac{|v|}{2}\right\}.$$ 
	    	In the region \(\left\{|v_1| \ge \frac{|v|}{2}\right\}\), we utilize the inequality
	    	$$ \sqrt{\mu_{1}} \le (\mu\mu_{1})^{\frac{1}{4}} = (\mu\mu_{1}\mu'\mu_{1}')^{\frac{1}{8}}.$$ 
	    	Accordingly, we deduce that
	    	\small
	    	\begin{align*}
	    		&\int_{\mathbb{S}^{2} \times \mathbb{R}^6} \mathbf{1}_{\left\{|v_1| \ge \frac{|v|}{2}\right\}}\sqrt{\mu_{1}} |u|^{\gamma_{2}}\, b_2 \, |g' h_1' \eta| \, d\omega dv_1dv  \\
	    		\lesssim&\int_{\mathbb{S}^{2} \times \mathbb{R}^6}  (\mu\mu_{1}\mu'\mu_{1}')^{\frac{1}{8}} |u|^{\gamma_{2}}\, b_2 \, |g' h_1' \eta| \, d\omega dv_1dv \\
	    		\lesssim&\int_{ \mathbb{R}^3} \bigg( \int_{\mathbb{S}^{2} \times \mathbb{R}^3} {\mu_{1}}^{\frac{1}{4}} |u|^{2\gamma_{2}}\, b_2d\omega dv_1\bigg)^{\frac{1}{2}} \bigg( \int_{\mathbb{S}^{2} \times \mathbb{R}^3} |(\mu^{\frac{1}{8}}g)' (\mu^{\frac{1}{8}}h)'_1|^{2} b_2d\omega dv_1\bigg)^{\frac{1}{2}} |\mu^{\frac{1}{8}}\eta|  dv\\
	    		\lesssim&  \bigg( \int_{\mathbb{S}^{2} \times \mathbb{R}^6} |(\mu^{\frac{1}{8}}g)' (\mu^{\frac{1}{8}}h)'_1|^{2} b_2d\omega dv_1dv\bigg)^{\frac{1}{2}} \|\mu^{\frac{1}{8}}\eta\|_{L^{2}_{v}}\\
	    		=& \bigg( \int_{\mathbb{S}^{2} \times \mathbb{R}^6} |(\mu^{\frac{1}{8}}g)' (\mu^{\frac{1}{8}}h)'_1 |^{2} b_2d\omega dv'_1dv'\bigg)^{\frac{1}{2}} \|\mu^{\frac{1}{8}}\eta\|_{L^{2}_{v}}\\
	    		\lesssim& \|\mu^{\frac{1}{8}}g\|_{L^{2}_{v}}\|\mu^{\frac{1}{8}}h\|_{L^{2}_{v}}\|\mu^{\frac{1}{8}}\eta\|_{L^{2}_{v}}.
	    	\end{align*}
	    	\normalsize 
	    	In the region \(\left\{|v_1| \le \frac{|v|}{2}\right\}\), invoking energy conservation \eqref{b conservation}, we obtain
	    	$$ |v|^{2} \gtrsim |v|^2 + |v_1|^2 = |v'|^2 + |v'_1|^2, $$
	    	which implies
	    	$$ \langle v \rangle^{\gamma_{2}} \lesssim \min \{ \langle v' \rangle^{\gamma_{2}}, \langle v'_{1} \rangle^{\gamma_{2}} \}. $$ 
	    	Consequently, we derive that
	    	\small
	    	\begin{align*}
	    		&\int_{\mathbb{S}^{2} \times \mathbb{R}^6} \mathbf{1}_{\left\{|v_1| \le \frac{|v|}{2}\right\}}\sqrt{\mu_{1}} |u|^{\gamma_{2}}\, b_2 \, |g' h_1' \eta| \, d\omega dv_1dv  \\ 
	    		\lesssim&\int_{ \mathbb{R}^3} \bigg( \int_{\mathbb{S}^{2} \times \mathbb{R}^3} {\mu_{1}} |u|^{2\gamma_{2}}\, b_2d\omega dv_1\bigg)^{\frac{1}{2}}  \bigg( \int_{\mathbb{S}^{2} \times \mathbb{R}^3} \mathbf{1}_{\left\{|v_1| \le \frac{|v|}{2}\right\}}|g' h'_1|^{2} b_2d\omega dv_1\bigg)^{\frac{1}{2}} |\eta|  dv\\
	    		\lesssim&  \bigg( \int_{\mathbb{S}^{2} \times \mathbb{R}^6}\langle v \rangle^{\gamma_{2}} \mathbf{1}_{\left\{|v_1| \le \frac{|v|}{2}\right\}}|g' h'_1 |^{2} b_2d\omega dv_1dv\bigg)^{\frac{1}{2}} \|\eta\|_{L^{2}_{\gamma_{2}}}\\
	    		\lesssim&  \bigg( \int_{\mathbb{S}^{2} \times \mathbb{R}^6}\min \{ \langle v' \rangle^{\gamma_{2}}, \langle v'_{1} \rangle^{\gamma_{2}} \}|g' h'_1 |^{2} b_2d\omega dv'_1dv'\bigg)^{\frac{1}{2}} \|\eta\|_{L^{2}_{\gamma_{2}}}\\
	    		\lesssim& \min \{ \|h\|_{L^{2}_{v}}\|g\|_{L^{2}_{\gamma_{2}}},\|g\|_{L^{2}_{v}}\|h\|_{L^{2}_{\gamma_{2}}}   \}\|\eta\|_{L^{2}_{\gamma_{2}}}.
	    	\end{align*}
	    	\normalsize 
	    	Combining both subdomains, we conclude that
	    	$$ |\langle\Gamma_{B,+}(g, h),\eta \rangle| \lesssim \min \{ \|h\|_{L^{2}_{v}}\|g\|_{L^{2}_{\gamma_{2}}},\|g\|_{L^{2}_{v}}\|h\|_{L^{2}_{\gamma_{2}}}   \}\|\eta\|_{L^{2}_{\gamma_{2}}}. $$
	    	
	    	Thus, the lemma has been established.
	    \end{proof}

		\begin{lemma}\label{Estimates of T}
			For \( \gamma_{3} \in (-3,1] \), the following estimates hold:
			$$ |\langle T_{-}(g, h,\xi),\eta \rangle|\lesssim \|h\|_{L^{2}_{v}}\|\xi\|_{L^{2}_{v}}\|g\|_{L^{2}_{\gamma_{3}}}\|\eta\|_{L^{2}_{\gamma_{3}}},$$
			and
			\small
			\begin{align*}
				|\langle T_{+}&(g, h,\xi),\eta \rangle|\\
				\lesssim&\big(\|g\|_{L^{2}_{v}}\|\xi\|_{L^{2}_{v}}\|h\|_{L^{2}_{\gamma_{3}}} +\|h\|_{L^{2}_{v}}\|\xi\|_{L^{2}_{v}}\|g\|_{L^{2}_{\gamma_{3}}}+\|g\|_{L^{2}_{v}}\|h\|_{L^{2}_{v}}\|\xi\|_{L^{2}_{\gamma_{3}}}\big)\|\eta\|_{L^{2}_{\gamma_{3}}}.
			\end{align*} 
		\normalsize
		\end{lemma}
	   \begin{proof}
	   	   Firstly, according to  the definitions of $B_3(\bm{u}, \bm{\omega})$ and $B_3(\bm{u}_1, \bm{\omega}) $ in \eqref{B_3},  we note that
	   	   \small
	   	   \begin{align*}
	   	   |\langle T_{-}(g, h,\xi), \eta \rangle| 
	   	   \le&\int_{\mathbb{S}^{5} \times \mathbb{R}^{9}} \sqrt{\mu_{1}\mu_2}  |\tilde{\bm{u}}|^{\gamma_3} \, b_3\left( \bar{\bm{u}} \cdot \bm{\omega}, \omega_1 \cdot \omega_2 \right) |g h_1 \xi_2 \eta| d\bm{\omega} dv_{1}dv_{2}dv\\
	   	   &+2\int_{\mathbb{S}^{5} \times \mathbb{R}^{9}} \sqrt{\mu_{1}\mu_2} |\tilde{\bm{u}}|^{\gamma_3} \, b_3\left( \bar{\bm{u}}_{1} \cdot \bm{\omega}, \omega_1 \cdot \omega_2 \right) |g h_1 \xi_2 \eta| d\bm{\omega} dv_{1}dv_{2}dv,
   	      \end{align*} 
         \normalsize
	   	   and
	   	   \small
	   	   \begin{align*}
	   	   	|\langle T_{+}(g, h,\xi), \eta \rangle| 
	   	   	\le&\int_{\mathbb{S}^{5} \times \mathbb{R}^{9}} \sqrt{\mu_{1}\mu_2}  |\tilde{\bm{u}}|^{\gamma_3} \, b_3\left( \bar{\bm{u}} \cdot \bm{\omega}, \omega_1 \cdot \omega_2 \right) |g^{*} h^{*}_1 \xi^{*}_2 \eta| d\bm{\omega} dv_{1}dv_{2}dv\\
	   	   	&+2\int_{\mathbb{S}^{5} \times \mathbb{R}^{9}} \sqrt{\mu_{1}\mu_2} |\tilde{\bm{u}}|^{\gamma_3} \, b_3\left( \bar{\bm{u}}_{1} \cdot \bm{\omega}, \omega_1 \cdot \omega_2 \right) |g^{1*} h^{1*}_1 \xi^{1*}_2 \eta| d\bm{\omega} dv_{1}dv_{2}dv.
	   	   \end{align*} 
	   	   \normalsize 
	   	   In the subsequent proof, we focus solely on the term containing \(b_3\left( \bar{\bm{u}} \cdot \bm{\omega}, \omega_1 \cdot \omega_2 \right)\), as the process can be analogously extended to the term involving \(b_3\left( \bar{\bm{u}}_{1} \cdot \bm{\omega}, \omega_1 \cdot \omega_2 \right)\). 
	   	   We commence the analysis by estimating the \( T_{-} \) term. 
	   	  Combining \eqref{tildeu-u}, \eqref{cutoff b3}, and Lemma \ref{Convolution estimates} with parameters $d=6$ and $q={\gamma}_{3}$, we obtain
	   	   \small
	   	   \begin{align*}
	   	   	 &\int_{\mathbb{S}^{5} \times \mathbb{R}^{9}} \sqrt{\mu_{1}\mu_2}  |\tilde{\bm{u}}|^{\gamma_3} \, b_3\left( \bar{\bm{u}} \cdot \bm{\omega}, \omega_1 \cdot \omega_2 \right) |g h_1 \xi_2 \eta| d\bm{\omega} dv_{1}dv_{2}dv\\ 
	   	   	 \lesssim&\int_{ \mathbb{R}^3} \bigg( \int_{ \mathbb{S}^{5} \times \mathbb{R}^6} {\mu_{1}}\mu_2 |\bm{u}|^{2\gamma_{3}}\, b_3d\bm{\omega} dv_{1}dv_{2}\bigg)^{\frac{1}{2}} \bigg( \int_{ \mathbb{S}^{5} \times \mathbb{R}^6} |h_1\xi_2|^{2} b_3d\bm{\omega} dv_{1}dv_{2}\bigg)^{\frac{1}{2}} |g \eta|  dv\\
	   	   	 \lesssim& \|h\|_{L^{2}_{v}}\|\xi\|_{L^{2}_{v}} \int_{ \mathbb{R}^3} \langle v\rangle^{\gamma_{3}} |g \eta|  dv\\
	   	   	 \lesssim& \|h\|_{L^{2}_{v}}\|\xi\|_{L^{2}_{v}}\|g\|_{L^{2}_{\gamma_{3}}}\|\eta\|_{L^{2}_{\gamma_{3}}}.
	   	   \end{align*} 
	   	   \normalsize 
	   	   We now turn to estimating the term $T_{+}$, which we divide into the following two cases.
	   	   
	   	   \noindent Case (i):  $\gamma_{3} \in [0,1]$.
	   	    From energy conservation \eqref{t conservation}, we observe that
	   	    $$\langle v\rangle^{\gamma_{3}}\lesssim\langle v^{*}\rangle^{\gamma_{3}}+\langle v^{*}_{1}\rangle^{\gamma_{3}}+\langle v^{*}_{2}\rangle^{\gamma_{3}}. $$
	   	    Accordingly, combining Lemma \ref{Convolution estimates} with the change of variables $T_{1,\bm{\omega}}$ satisfying $|\det T_{1,\bm{\omega}}| = 1$, we conclude that 
	   	   \small
	   	   \begin{align*}
	   	   	&\int_{\mathbb{S}^{5} \times \mathbb{R}^{9}} \sqrt{\mu_{1}\mu_2}  |\tilde{\bm{u}}|^{\gamma_3} \, b_3\left( \bar{\bm{u}} \cdot \bm{\omega}, \omega_1 \cdot \omega_2 \right) |g^{*} h^{*}_1 \xi^{*}_2 \eta| d\bm{\omega} dv_{1}dv_{2}dv\\ 
	   	   	\lesssim&\int_{ \mathbb{R}^3} \bigg( \int_{ \mathbb{S}^{5} \times \mathbb{R}^6} {\mu_{1}}\mu_2 |\bm{u}|^{2\gamma_{3}}\, b_3d\bm{\omega} dv_{1}dv_{2}\bigg)^{\frac{1}{2}} \bigg( \int_{ \mathbb{S}^{5} \times \mathbb{R}^6} |g^{*}h^{*}_1\xi^{*}_2|^{2} b_3d\bm{\omega} dv_{1}dv_{2}\bigg)^{\frac{1}{2}} | \eta|  dv\\
	   	    \lesssim&\int_{ \mathbb{R}^3} \langle v\rangle^{\gamma_{3}} \bigg( \int_{ \mathbb{S}^{5} \times \mathbb{R}^6} |g^{*}h^{*}_1\xi^{*}_2|^{2} b_3d\bm{\omega} dv_{1}dv_{2}\bigg)^{\frac{1}{2}} | \eta|  dv\\
	   	    \lesssim&  \bigg( \int_{ \mathbb{S}^{5} \times \mathbb{R}^9} \langle v\rangle^{\gamma_{3}}|g^{*}h^{*}_1\xi^{*}_2|^{2} b_3d\bm{\omega} dv_{1}dv_{2}dv\bigg)^{\frac{1}{2}}\|\eta\|_{L^{2}_{\gamma_{3}}}\\
	   	   	\lesssim&\bigg( \int_{ \mathbb{S}^{5} \times \mathbb{R}^9} (\langle v^{*}\rangle^{\gamma_{3}}+\langle v^{*}_{1}\rangle^{\gamma_{3}}+\langle v^{*}_{2}\rangle^{\gamma_{3}})|g^{*}h^{*}_1\xi^{*}_2|^{2} b_3d\bm{\omega} dv^{*}_{1}dv^{*}_{2}dv^{*}\bigg)^{\frac{1}{2}}\|\eta\|_{L^{2}_{\gamma_{3}}}\\
	   	   	\lesssim& \big(\|g\|_{L^{2}_{v}}\|\xi\|_{L^{2}_{v}}\|h\|_{L^{2}_{\gamma_{3}}} +\|h\|_{L^{2}_{v}}\|\xi\|_{L^{2}_{v}}\|g\|_{L^{2}_{\gamma_{3}}}+\|g\|_{L^{2}_{v}}\|h\|_{L^{2}_{v}}\|\xi\|_{L^{2}_{\gamma_{3}}}\big)\|\eta\|_{L^{2}_{\gamma_{3}}}.
	   	   \end{align*} 
	   	   \normalsize

	   	   \noindent Case (ii):  $\gamma_{3} \in (-3,0)$.
	   	   In the domain of integration, for any $(v, v_1, v_2) \in \mathbb{R}^9$, we define the vectors 
	   	   $$
	   	   \bm{y} := \begin{pmatrix} v_1 \\ v_2 \end{pmatrix}, \quad \bm{v} := \begin{pmatrix} v \\ v \end{pmatrix}.
	   	   $$ 
	   	   It is evident that $\bm{u} = \bm{y} - \bm{v}$. We partition the integration domain into two regions:
	   	   $$
	   	   \left\{ |\bm{y}| \ge \frac{|\bm{v}|}{2} \right\} \cup \left\{ |\bm{y}| \le \frac{|\bm{v}|}{2} \right\}.
	   	   $$
	   	   In the region $\left\{ |\bm{y}| \ge \frac{|\bm{v}|}{2} \right\}$, we observe that
	   	   $$ \sqrt{\mu_{1}\mu_2}\lesssim(\mu\mu_{1}\mu_2)^{\frac{1}{16}}=(\mu\mu_{1}\mu_2\mu^{*}\mu^{*}_{1}\mu^{*}_2)^{\frac{1}{32}}.$$ 
	   	   From this, and incorporating the change of variables \( T_{1,\bm{\omega}} \), we deduce that
	   	   \small
	   	   \begin{align*}
	   	   	&\int_{\mathbb{S}^{5} \times \mathbb{R}^{9}} \mathbf{1}_{\left\{ |\bm{y}| \ge \frac{|\bm{v}|}{2} \right\}} \sqrt{\mu_{1}\mu_2}  |\tilde{\bm{u}}|^{\gamma_3} \, b_3\left( \bar{\bm{u}} \cdot \bm{\omega}, \omega_1 \cdot \omega_2 \right) |g^{*} h^{*}_1 \xi^{*}_2 \eta| d\bm{\omega} dv_{1}dv_{2}dv\\ 
	   	   	\lesssim&\int_{ \mathbb{R}^3} \bigg( \int_{ \mathbb{S}^{5} \times \mathbb{R}^6} ({\mu_{1}}\mu_2)^\frac{1}{16} |\bm{u}|^{2\gamma_{3}}\, b_3d\bm{\omega} dv_{1}dv_{2}\bigg)^{\frac{1}{2}}\\
	   	   	&\quad\quad\times \bigg( \int_{ \mathbb{S}^{5} \times \mathbb{R}^6} |(\mu^{\frac{1}{32}}g)^{*}(\mu^{\frac{1}{32}}h)^{*}_1(\mu^{\frac{1}{32}}\xi)^{*}_2|^{2} b_3d\bm{\omega} dv_{1}dv_{2}\bigg)^{\frac{1}{2}} |\mu^{\frac{1}{32}} \eta|  dv\\
	   	   	\lesssim&\int_{ \mathbb{R}^3} \bigg( \int_{ \mathbb{S}^{5} \times \mathbb{R}^6} |(\mu^{\frac{1}{32}}g)^{*}(\mu^{\frac{1}{32}}h)^{*}_1(\mu^{\frac{1}{32}}\xi)^{*}_2|^{2} b_3d\bm{\omega} dv_{1}dv_{2}\bigg)^{\frac{1}{2}} |\mu^{\frac{1}{32}} \eta|  dv\\ 
	   	   	\lesssim&\bigg( \int_{ \mathbb{S}^{5} \times \mathbb{R}^9} |(\mu^{\frac{1}{32}}g)^{*}(\mu^{\frac{1}{32}}h)^{*}_1(\mu^{\frac{1}{32}}\xi)^{*}_2|^{2} b_3d\bm{\omega} dv^{*}_{1}dv^{*}_{2}dv^{*}\bigg)^{\frac{1}{2}}\|\mu^{\frac{1}{32}}\eta\|_{L^{2}_{v}}\\
	   	   	\lesssim& \|\mu^{\frac{1}{32}}g\|_{L^{2}_{v}}\|\mu^{\frac{1}{32}}h\|_{L^{2}_{v}}\|\mu^{\frac{1}{32}}\xi\|_{L^{2}_{v}}\|\mu^{\frac{1}{32}}\eta\|_{L^{2}_{v}}.
	   	   \end{align*} 
	   	   \normalsize  
	   	   In the region $\left\{ |\bm{y}| \le \frac{|\bm{v}|}{2} \right\}$, by energy conservation in \eqref{t conservation}, it follows that
	   	   $$ |v|^{2}=2|\bm{v}|^{2} \gtrsim|v|^2+ |\bm{y}|^2=|v|^2 + |v_1|^2 + |v_2|^2= |v^{*}|^2 + |v^{*}_1|^2+ |v^{*}_2|^2, $$
	   	   from which we derive that
	   	   $$ \langle v \rangle^{\gamma_{3}} \lesssim \min \{ \langle v^{*} \rangle^{\gamma_{3}}, \langle v^{*}_{1} \rangle^{\gamma_{3}},  \langle v^{*}_{2} \rangle^{\gamma_{3}}\}. $$  
	   	   Therefore, we conclude that 
	   	    \small
	   	    \begin{align*}
	   	    	&\int_{\mathbb{S}^{5} \times \mathbb{R}^{9}} \mathbf{1}_{\left\{ |\bm{y}| \le \frac{|\bm{v}|}{2} \right\}}\sqrt{\mu_{1}\mu_2}  |\tilde{\bm{u}}|^{\gamma_3} \, b_3\left( \bar{\bm{u}} \cdot \bm{\omega}, \omega_1 \cdot \omega_2 \right) |g^{*} h^{*}_1 \xi^{*}_2 \eta| d\bm{\omega} dv_{1}dv_{2}dv\\ 
	   	    	\lesssim&\int_{ \mathbb{R}^3} \bigg( \int_{ \mathbb{S}^{5} \times \mathbb{R}^6} {\mu_{1}}\mu_2 |\bm{u}|^{2\gamma_{3}}\, b_3d\bm{\omega} dv_{1}dv_{2}\bigg)^{\frac{1}{2}}\\
	   	    	&\quad\quad\times \bigg( \int_{ \mathbb{S}^{5} \times \mathbb{R}^6} \mathbf{1}_{\left\{ |\bm{y}| \le \frac{|\bm{v}|}{2} \right\}}|g^{*}h^{*}_1\xi^{*}_2|^{2} b_3d\bm{\omega} dv_{1}dv_{2}\bigg)^{\frac{1}{2}} | \eta|  dv\\
	   	    	\lesssim&\int_{ \mathbb{R}^3} \langle v\rangle^{\gamma_{3}} \bigg( \int_{ \mathbb{S}^{5} \times \mathbb{R}^6}\mathbf{1}_{\left\{ |\bm{y}| \le \frac{|\bm{v}|}{2} \right\}} |g^{*}h^{*}_1\xi^{*}_2|^{2} b_3d\bm{\omega} dv_{1}dv_{2}\bigg)^{\frac{1}{2}} | \eta|  dv\\
	   	    	\lesssim&  \bigg( \int_{ \mathbb{S}^{5} \times \mathbb{R}^9} \langle v\rangle^{\gamma_{3}}\mathbf{1}_{\left\{ |\bm{y}| \le \frac{|\bm{v}|}{2} \right\}}|g^{*}h^{*}_1\xi^{*}_2|^{2} b_3d\bm{\omega} dv_{1}dv_{2}dv\bigg)^{\frac{1}{2}}\|\eta\|_{L^{2}_{\gamma_{3}}}\\
	   	    	\lesssim&\bigg( \int_{ \mathbb{S}^{5} \times \mathbb{R}^9} \min \{ \langle v^{*} \rangle^{\gamma_{3}}, \langle v^{*}_{1} \rangle^{\gamma_{3}},  \langle v^{*}_{2} \rangle^{\gamma_{3}}\}|g^{*}h^{*}_1\xi^{*}_2|^{2} b_3d\bm{\omega} dv^{*}_{1}dv^{*}_{2}dv^{*}\bigg)^{\frac{1}{2}}\|\eta\|_{L^{2}_{\gamma_{3}}}\\
	   	    	\lesssim& \min \{ \|g\|_{L^{2}_{v}}\|\xi\|_{L^{2}_{v}}\|h\|_{L^{2}_{\gamma_{3}}} ,\|h\|_{L^{2}_{v}}\|\xi\|_{L^{2}_{v}}\|g\|_{L^{2}_{\gamma_{3}}},\|g\|_{L^{2}_{v}}\|h\|_{L^{2}_{v}}\|\xi\|_{L^{2}_{\gamma_{3}}}\} \|\eta\|_{L^{2}_{\gamma_{3}}}.
	   	    \end{align*} 
	   	    \normalsize 
	   	    
	   	   Hence, the lemma is proved. 
	   \end{proof}

		\begin{corollary}\label{Estimates of Gamma T mu}
			Let \( \gamma_{2} \in (-\frac{3}{2},1] \) and \( \gamma_{3} \in (-3,1] \). For a function \( \phi \) satisfying \( |\phi| \lesssim \mu^{\rho} \) for some \( \rho > 0 \), there exists a \( \delta > 0 \) such that
			$$ |\langle \Gamma_{B,\pm}(g, h), \phi \rangle| \lesssim \|\mu^{\delta} g\|_{L^{2}_{v}} \|\mu^{\delta} h\|_{L^{2}_{v}} , $$
			and
			$$ |\langle T_{\pm}(g, h, \xi), \phi \rangle| \lesssim \|\mu^{\delta} g\|_{L^{2}_{v}} \|\mu^{\delta} h\|_{L^{2}_{v}} \|\mu^{\delta} \xi\|_{L^{2}_{v}} . $$
		\end{corollary}
	     \begin{proof} Without loss of generality, we may assume $\rho < \tfrac12$.  Then, from \eqref{b conservation mu} and \eqref{t conservation mu}, by taking $\delta = \tfrac{\rho}{4}$, it is straightforward to see that 
	     	\small
	     	\begin{align*}
	     		|\langle\Gamma_{B,-}(g, h),\phi \rangle| \lesssim&\int_{\mathbb{S}^{2} \times \mathbb{R}^6} \sqrt{\mu_{1}}\mu^{\rho} |u|^{\gamma_{2}}\, b_2 \, |g h_1 | \, d\omega dv_1dv  \\
	     		\lesssim&\int_{\mathbb{S}^{2} \times \mathbb{R}^6} (\mu\mu_{1})^{2\delta } |u|^{\gamma_{2}}\, b_2 \, |(\mu^{2\delta }g) (\mu^{2\delta }h)_1 | \, d\omega dv_1dv ,
	     	\end{align*}
	     	\normalsize 
	     	\small
	     	\begin{align*}
	     		|\langle\Gamma_{B,+}(g, h),\phi \rangle| \lesssim&\int_{\mathbb{S}^{2} \times \mathbb{R}^6} \sqrt{\mu_{1}}\mu^{\rho} |u|^{\gamma_{2}}\, b_2 \, |g' h'_1 | \, d\omega dv_1dv  \\
	     		\lesssim&\int_{\mathbb{S}^{2} \times \mathbb{R}^6} (\mu\mu_{1})^{2\delta } |u|^{\gamma_{2}}\, b_2 \, |(\mu^{2\delta }g)' (\mu^{2\delta }h)'_1 | \, d\omega dv_1dv ,
	     	\end{align*}
	     	\normalsize 
	     	\small
	     	\begin{align*}
	     		&|\langle T_{-}(g, h,\xi), \phi\rangle| \\
	     		\lesssim&\int_{\mathbb{S}^{5} \times \mathbb{R}^{9}} \sqrt{\mu_{1}\mu_2}\mu^{\rho}  |\tilde{\bm{u}}|^{\gamma_3} \, b_3\left( \bar{\bm{u}} \cdot \bm{\omega}, \omega_1 \cdot \omega_2 \right) |g h_1 \xi_2| d\bm{\omega} dv_{1}dv_{2}dv\\
	     		&+2\int_{\mathbb{S}^{5} \times \mathbb{R}^{9}} \sqrt{\mu_{1}\mu_2}\mu^{\rho}  |\tilde{\bm{u}}|^{\gamma_3} \, b_3\left( \bar{\bm{u}}_{1} \cdot \bm{\omega}, \omega_1 \cdot \omega_2 \right) |g h_1 \xi_2 | d\bm{\omega} dv_{1}dv_{2}dv\\
	     		\lesssim&\int_{\mathbb{S}^{5} \times \mathbb{R}^{9}} (\mu\mu_{1}\mu_{2})^{2\delta }   |\tilde{\bm{u}}|^{\gamma_3} \, b_3\left( \bar{\bm{u}} \cdot \bm{\omega}, \omega_1 \cdot \omega_2 \right) |(\mu^{2\delta }g) (\mu^{2\delta }h)_1 (\mu^{2\delta }\xi)_2| d\bm{\omega} dv_{1}dv_{2}dv\\
	     		&+2\int_{\mathbb{S}^{5} \times \mathbb{R}^{9}} (\mu\mu_{1}\mu_{2})^{2\delta }  |\tilde{\bm{u}}|^{\gamma_3} \, b_3\left( \bar{\bm{u}}_{1} \cdot \bm{\omega}, \omega_1 \cdot \omega_2 \right) |(\mu^{2\delta }g) (\mu^{2\delta }h)_1 (\mu^{2\delta }\xi)_2|d\bm{\omega} dv_{1}dv_{2}dv,
	     	\end{align*} 
	     	\normalsize
	     	and
	     	\small
	     	\begin{align*}
	     		&|\langle T_{+}(g, h,\xi), \phi\rangle| \\ 
	     		\lesssim&\int_{\mathbb{S}^{5} \times \mathbb{R}^{9}} \sqrt{\mu_{1}\mu_2}\mu^{\rho}  |\tilde{\bm{u}}|^{\gamma_3} \, b_3\left( \bar{\bm{u}} \cdot \bm{\omega}, \omega_1 \cdot \omega_2 \right) |g^{*} h^{*}_1 \xi^{*}_2 | d\bm{\omega} dv_{1}dv_{2}dv\\
	     		&+2\int_{\mathbb{S}^{5} \times \mathbb{R}^{9}} \sqrt{\mu_{1}\mu_2}\mu^{\rho}  |\tilde{\bm{u}}|^{\gamma_3} \, b_3\left( \bar{\bm{u}}_{1} \cdot \bm{\omega}, \omega_1 \cdot \omega_2 \right) |g^{1*} h^{1*}_1 \xi^{1*}_2 |d\bm{\omega} dv_{1}dv_{2}dv\\
	     		\lesssim&\int_{\mathbb{S}^{5} \times \mathbb{R}^{9}} (\mu\mu_{1}\mu_{2})^{2\delta }   |\tilde{\bm{u}}|^{\gamma_3} \, b_3\left( \bar{\bm{u}} \cdot \bm{\omega}, \omega_1 \cdot \omega_2 \right) |(\mu^{2\delta }g)^{*}  (\mu^{2\delta }h)^{*} _1 (\mu^{2\delta }\xi)^{*} _2| d\bm{\omega} dv_{1}dv_{2}dv\\
	     		&+2\int_{\mathbb{S}^{5} \times \mathbb{R}^{9}} (\mu\mu_{1}\mu_{2})^{2\delta }  |\tilde{\bm{u}}|^{\gamma_3} \, b_3\left( \bar{\bm{u}}_{1} \cdot \bm{\omega}, \omega_1 \cdot \omega_2 \right) |(\mu^{2\delta }g)^{1*} (\mu^{2\delta }h)^{1*}_1 (\mu^{2\delta }\xi)^{1*}_2|d\bm{\omega} dv_{1}dv_{2}dv.
	     	\end{align*}  
	     	\normalsize 
	     	Furthermore, by following the proof steps of Lemma \ref{Estimates of GammaB} and Lemma \ref{Estimates of T}, the corollary is established. 
	     \end{proof}
		
		\begin{corollary}\label{Estimates of GammaT}
			For \( \gamma_3 \in (-3, 1] \), the following bound holds:  
			$$ |\langle\Gamma_{T}(g, h),\eta \rangle|\lesssim \big(\|g\|_{L^{2}_{v}}\|h\|_{L^{2}_{\gamma_{3}}}+\|h\|_{L^{2}_{v}}\|g\|_{L^{2}_{\gamma_{3}}}\big)\|\eta\|_{L^{2}_{\gamma_{3}}}.$$
		\end{corollary}
		\begin{proof} 
			From the expression of $\Gamma_{T}(g, h)$ in  \eqref{Gamma_{{T}}(g,h)}, we observe that
			\begin{align*} 
				\Gamma_{{T}}(g,h) = &T_{+}(g, h, \sqrt{\mu}) + T_{+}(g, \sqrt{\mu}, h) + T_{+}(\sqrt{\mu}, g, h)\\
				&-T_{-}(g, h, \sqrt{\mu}) - T_{-}(g, \sqrt{\mu}, h) - T_{-}(\sqrt{\mu}, g, h).
			\end{align*} 
			
			\noindent{(i)} Estimate for $ T_{+}(g, h, \sqrt{\mu}), \ T_{+}(g, \sqrt{\mu}, h)$ and $ T_{+}(\sqrt{\mu}, g, h). $
			
			By adhering to the methodology outlined in the proof of Lemma \ref{Estimates of T}, we derive the subsequent estimates.
			When $\gamma_3 \in [0, 1]$, we have 
			\begin{align*} 
				&|\langle T_{+}(g, h, \sqrt{\mu}),\eta\rangle|+|\langle T_{+}(g, \sqrt{\mu}, h),\eta\rangle|+|\langle T_{+}(\sqrt{\mu}, g, h),\eta\rangle|\\[1mm]
				\lesssim& \big(\|g\|_{L^{2}_{v}}\|h\|_{L^{2}_{\gamma_{3}}} +\|h\|_{L^{2}_{v}}\|g\|_{L^{2}_{\gamma_{3}}}+\|g\|_{L^{2}_{v}}\|h\|_{L^{2}_{v}}\big)\|\eta\|_{L^{2}_{\gamma_{3}}}\\[1mm]
				\lesssim&\big(\|g\|_{L^{2}_{v}}\|h\|_{L^{2}_{\gamma_{3}}} +\|h\|_{L^{2}_{v}}\|g\|_{L^{2}_{\gamma_{3}}}\big)\|\eta\|_{L^{2}_{\gamma_{3}}},
			\end{align*} 
			whereas for $\gamma_3 \in (-3, 0)$, we obtain 
			\begin{align*} 
				&|\langle T_{+}(g, h, \sqrt{\mu}),\eta\rangle|+|\langle T_{+}(g, \sqrt{\mu}, h),\eta\rangle|+|\langle T_{+}(\sqrt{\mu}, g, h),\eta\rangle|\\[1mm]
				\lesssim& \min \{ \|g\|_{L^{2}_{v}}\|h\|_{L^{2}_{\gamma_{3}}} ,\|h\|_{L^{2}_{v}}\|g\|_{L^{2}_{\gamma_{3}}},\|g\|_{L^{2}_{v}}\|h\|_{L^{2}_{v}}\} \|\eta\|_{L^{2}_{\gamma_{3}}}\\[1mm]
				&+\|\mu^{\frac{1}{32}}g\|_{L^{2}_{v}}\|\mu^{\frac{1}{32}}h\|_{L^{2}_{v}}\|\mu^{\frac{1}{32}}\eta\|_{L^{2}_{v}} \\[1mm]
				\lesssim&\|g\|_{L^{2}_{v}}\|h\|_{L^{2}_{\gamma_{3}}}\|\eta\|_{L^{2}_{\gamma_{3}}}.
			\end{align*} 
			
			\noindent{(ii)} Estimate for $ T_{-}(g, h, \sqrt{\mu})$ and $T_{-}(g, \sqrt{\mu}, h). $
			
			From Lemma \ref{Estimates of T}, it is evident that the estimate for $|\langle T_{-}(g, h, \xi), \eta \rangle|$ involves only the $L^{2}_{\gamma_{3}}$ norms of the functions $g$ and $\eta.$ 
			Specifically, we have
			$$
			|\langle T_{-}(g, h, \xi), \eta \rangle| \lesssim \|h\|_{L^{2}_{v}} \|\xi\|_{L^{2}_{v}} \|g\|_{L^{2}_{\gamma_{3}}} \|\eta\|_{L^{2}_{\gamma_{3}}}.
			$$ 
			Consequently, we derive that
			$$
			|\langle T_{-}(g, h, \sqrt{\mu}), \eta \rangle|+ |\langle T_{-}(g, \sqrt{\mu}, h), \eta \rangle|\lesssim \|h\|_{L^{2}_{v}}  \|g\|_{L^{2}_{\gamma_{3}}} \|\eta\|_{L^{2}_{\gamma_{3}}}.
			$$ 
			
			\noindent{(iii)} Estimate for $T_{-}(\sqrt{\mu},g, h). $	
			
			By analogy with the proof of Corollary \ref{Estimates of Gamma T mu}, we can similarly deduce that
			$$
			|\langle T_{-}(\sqrt{\mu},g, h), \eta \rangle|=|\langle T_{-}(\eta ,g, h), \sqrt{\mu} \rangle| \lesssim\|\mu^{\frac{1}{8}} g\|_{L^{2}_{v}} \|\mu^{\frac{1}{8}} h\|_{L^{2}_{v}} \|\mu^{\frac{1}{8}} \eta\|_{L^{2}_{v}}.
			$$ 
			
			Consequently, the corollary follows from the aforementioned estimates.
			
		\end{proof}

		In the following corollary, we establish weighted \(L^2_v\) estimates for certain linear operators. For hard potentials, the weighted estimates below do not establish the boundedness of $K_T$ on the unweighted space $L^2_v$.
		
\begin{corollary}\label{Estimates of LK}
			Let \( \gamma_{2} \in (-\frac{3}{2},1] \) and \( \gamma_{3} \in (-3,1] \). We have
			$$ |\langle L_{B}g,h\rangle|+|\langle K_{B}g,h\rangle|\lesssim  \|g\|_{L^{2}_{\gamma_{2}}}\|h\|_{L^{2}_{\gamma_{2}}},$$
			and
			$$ |\langle L_{T}g,h\rangle|+|\langle K_{T}g,h\rangle|\lesssim  \|g\|_{L^{2}_{\gamma_{3}}}\|h\|_{L^{2}_{\gamma_{3}}}.$$
			Moreover, for the operator \( K_B \) we have the estimate
			\[
			|\langle K_B g, h \rangle| \lesssim \|g\|_{L^2_{v}}\|h\|_{L^2_{v}}.
			\]
		\end{corollary}
		\begin{proof}
			We begin with the proof concerning the operator \(L_T\). Recall that  
			\begin{align*} 
				L_T g =& T_{-}(g, \sqrt{\mu}, \sqrt{\mu}) + T_{-}(\sqrt{\mu}, g, \sqrt{\mu}) + T_{-}(\sqrt{\mu}, \sqrt{\mu}, g)\\[1mm]
				 &- T_{+}(g, \sqrt{\mu}, \sqrt{\mu}) - T_{+}(\sqrt{\mu}, g, \sqrt{\mu}) - T_{+}(\sqrt{\mu}, \sqrt{\mu}, g).
			\end{align*}  
			If \(\gamma_3 \in [0,1]\), we directly utilize the estimates from Lemma \ref{Estimates of T} to obtain  
			\[
			|\langle L_{T}g, h \rangle| \lesssim \left( \|g\|_{L^{2}_{\gamma_{3}}} + \|g\|_{L^{2}_{v}} \right) \|h\|_{L^{2}_{\gamma_{3}}} \lesssim \|g\|_{L^{2}_{\gamma_{3}}} \|h\|_{L^{2}_{\gamma_{3}}}.
			\] 
			If \(\gamma_3 \in (-3,0)\), it follows from the proof of Lemma \ref{Estimates of T} that  
			$$ |\langle T_{-}(g, h,\xi),\eta \rangle|\lesssim \|h\|_{L^{2}_{v}}\|\xi\|_{L^{2}_{v}}\|g\|_{L^{2}_{\gamma_{3}}}\|\eta\|_{L^{2}_{\gamma_{3}}},$$
			and
			\[
			\begin{aligned}
				|\langle T_{+}(g, h, \xi), \eta \rangle| \lesssim& \min \big\{  \|g\|_{L^{2}_{v}} \|\xi\|_{L^{2}_{v}} \|h\|_{L^{2}_{\gamma_{3}}},\; \|h\|_{L^{2}_{v}} \|\xi\|_{L^{2}_{v}} \|g\|_{L^{2}_{\gamma_{3}}}, \\
				&\qquad\quad \|g\|_{L^{2}_{v}} \|h\|_{L^{2}_{v}} \|\xi\|_{L^{2}_{\gamma_{3}}} \big\} \|\eta\|_{L^{2}_{\gamma_{3}}}\\
				&+\|\mu^{\frac{1}{32}}g\|_{L^{2}_{v}}\|\mu^{\frac{1}{32}}h\|_{L^{2}_{v}}\|\mu^{\frac{1}{32}}\xi\|_{L^{2}_{v}}\|\mu^{\frac{1}{32}}\eta\|_{L^{2}_{v}},
			\end{aligned}
			\]  
			from which we deduce that  
			$$|\langle T_{-}(g, \sqrt{\mu}, \sqrt{\mu}), h \rangle|\lesssim \|g\|_{L^{2}_{\gamma_{3}}} \|h\|_{L^{2}_{\gamma_{3}}}, $$
			and 
			\begin{align*}
				& |\langle T_{+}(g, \sqrt{\mu}, \sqrt{\mu}), h \rangle| + |\langle T_{+}(\sqrt{\mu}, g, \sqrt{\mu}), h \rangle| + |\langle T_{+}(\sqrt{\mu}, \sqrt{\mu}, g), h \rangle| \\
				& \qquad \lesssim \|g\|_{L^{2}_{\gamma_{3}}} \|h\|_{L^{2}_{\gamma_{3}}}.
			\end{align*} 
			The remaining two terms can be estimated using Corollary \ref{Estimates of Gamma T mu}, which yields the following bound:
			\begin{align*}
				& |\langle T_{-}(\sqrt{\mu}, g, \sqrt{\mu}), h \rangle| + |\langle T_{-}(\sqrt{\mu}, \sqrt{\mu}, g), h \rangle| \\
				=& |\langle T_{-}(h, g, \sqrt{\mu}), \sqrt{\mu} \rangle| + |\langle T_{-}(h, \sqrt{\mu}, g), \sqrt{\mu} \rangle|\\
				\lesssim&\|g\|_{L^{2}_{\gamma_{3}}} \|h\|_{L^{2}_{\gamma_{3}}}.
			\end{align*} 
			Therefore, we have  
			\[
			|\langle L_{T}g, h\rangle| \lesssim \|g\|_{L^{2}_{\gamma_{3}}} \|h\|_{L^{2}_{\gamma_{3}}}.
			\]  
			Combining \eqref{L_T} and \eqref{nu3}, we infer that  
			\[
			|\langle K_{T}g, h\rangle| \lesssim |\langle L_{T}g, h\rangle| + |\langle \nu_{T}g, h\rangle| \lesssim \|g\|_{L^{2}_{\gamma_{3}}} \|h\|_{L^{2}_{\gamma_{3}}}.
			\]
			
			Furthermore, combining \eqref{eq:LB_decomp}, \eqref{nuB}, and Lemma \ref{Estimates of GammaB}, a similar procedure yields the estimates
			\[
			|\langle L_{B}g, h\rangle| + |\langle K_{B}g, h\rangle| \lesssim \|g\|_{L^{2}_{\gamma_{2}}} \|h\|_{L^{2}_{\gamma_{2}}}.
			\]
			For the case \( \gamma_{2} \in (-\frac{3}{2}, 0] \), it is clear that
			\[
			|\langle K_{B}g, h\rangle| \lesssim \|g\|_{L^{2}_{\gamma_{2}}} \|h\|_{L^{2}_{\gamma_{2}}} \lesssim \|g\|_{L^{2}_{v}} \|h\|_{L^{2}_{v}}.
			\]
			For \( \gamma_{2} \in (0, 1] \), it can be shown that \(K_{B}\) is a compact operator on \(L^{2}_{v}\) (refer to \cite{Glassey1996}, Section 3.5), and is hence bounded on \(L^{2}_{v}\). 
		\end{proof}
		
		\section{Global Existence}\label{sec;Global Existence}
		
		In this section, we focus on the proof of the main result, Theorem \ref{main_theorem}. We first establish the local well-posedness and then extend the local solution to a global one by combining it with the macroscopic estimates.

\paragraph{4.1. Local Existence.}
		
		The proof of local well-posedness can be divided into two main steps. First, we establish the local existence, uniqueness, and continuity in time of the energy functional for the solution. Then, we demonstrate the non-negativity of the solution. It is important to note that, when using the positivity-preserving iteration to prove non-negativity, the available weighted estimates for $K_T$ do not directly close this iteration for hard potentials, so we introduce a truncation and limiting procedure on the ternary collision kernel.
		
		\begin{theorem}\label{local existence} 
			Let \( N\ge2 \), \( \gamma_2 \in (-\frac{3}{2}, 1] \), \( \gamma_3 \in (-3, 1] \), and assume that \( \gamma_3 \leq \gamma_2 \). 
			Suppose the initial data for \eqref{BT-g} satisfies
			$f_0(x,v) = \mu(v) + \sqrt{\mu(v)} g_0(x,v) \geq 0.$
			Then, there exist constants \( T^{*}, M_{0} \in (0,1) \) such that if the initial perturbation satisfies  
			\[ \|g_0\|^{2}_{H^N_x L^2_v} \leq {M_{0}}, \]
			the binary-ternary Boltzmann equation \eqref{BT-g} admits a unique local-in-time  solution \( g(t,x,v) \) in \([0, T^{*}]\times\mathbb{R}^{6}_{x,v}\) satisfying:
			\begin{enumerate}
				\item[(i)] \( f(t,x,v) = \mu(v) + \sqrt{\mu(v)} \, g(t,x,v) \geq 0 \);
				\item[(ii)] \(\displaystyle \|g(t)\|^{2}_{H^N_x L^2_v} + \int_{0}^{t}\|g(s)\|^{2}_{H^N_x L^2_\mathcal{D}} ds \le 2\|g_0\|^{2}_{H^N_x L^2_v},\)$\quad\forall\ t \in [0, T^{*}]$;
				\item[(iii)] \(\|g(t)\|_{H^N_x L^2_v}\) is continuous on \([0, T^{*}]\).
			\end{enumerate}
		\end{theorem}
	     \begin{proof}
	 For notational convenience in the proof of this theorem, we set 
	 \[
	 \|g(t)\|^{2}_{E}:=\|g(t)\|^{2}_{H^{N}_{x}L^{2}_{v}}, \qquad 
	 \|g(t)\|^{2}_{D}:=\|g(t)\|^{2}_{H^{N}_{x}L^{2}_{\mathcal{D}}},
	 \] 
	 and define the total energy functional 
	 \[
	 \mathbb{E}_{T}(g):=\sup_{t\in[0,T]}\|g(t)\|^{2}_{E}+\int_{0}^{T}\|g(t)\|^{2}_{D}\,dt.
	 \]    	
	   
	  \noindent \textit{Step 1. Local well-posedness.}
	     	
	   We establish the existence and uniqueness of a solution  to  equation \eqref{BT-g} on a sufficiently small time interval \([0, T^*]\).
	   Specifically, we adopt the following linearized iterative scheme:   
	   \begin{align}\label{iterative scheme g 1}
	   	\begin{aligned}
	   		\partial_t g^{n+1} &+ v \cdot \nabla_x g^{n+1} + \nu_{B}(v)g^{n+1} + L_{T}g^{n+1} \\  
	   		&= K_{B}g^{n} + \Gamma_{B}(g^{n}, g^{n}) + \Gamma_{T}(g^{n}, g^{n}) + T(g^{n}, g^{n}, g^{n}),
	   	\end{aligned} 
	   \end{align}  
	   with initial data \(g^{n+1}(0) = g_0\) and starting from \(g^0 = 0\).  
      First, we claim that the sequence $ g^{n} $  admits a uniform estimate.
      \begin{lemma}\label{uniform estimate gn}
      	There exist sufficiently small parameters $T^* > 0$ and $M_0 > 0$ such that  
      	\[
      	\sup_{n \geq 0} \mathbb{E}_{T^*}(g^n) \leq 2 \|g_0\|_E^2.
      	\]
      \end{lemma}
	  \begin{proof}
	  	We proceed by induction.  
	  	First, we observe that for any \(T^{*}\in(0,1)\), 
	  	\(
	  	\mathbb{E}_{T^{*}}(g^{1})\le 2\,\|g_{0}\|_{E}^{2},
	  	\)
	  	since \(g^{0}\equiv 0\) and $ L_T $ is positive. 
	  	Next, we assume that for some \(T^{*}\in(0,1)\),  
	  	\begin{equation}\label{induction gn}
	  		\mathbb{E}_{T^{*}}(g^{n})\le 2\,\|g_{0}\|_{E}^{2}.
	  	\end{equation} 
	  	We shall prove that the same bound holds for \(g^{n+1}\). 
	  	This is established through the energy estimate procedure.
	  	Applying the spatial derivative operator \(\partial^\alpha\) to equation \eqref{iterative scheme g 1}, multiplying by \(\partial^\alpha g^{n+1}\), integrating over \(x\) and \(v\), and summing over all multi-indices \(\alpha\) with \(|\alpha| \leq N\) (\(N\ge2\)) yields:  
	  	\begin{align*}
	  		\frac{1}{2} \frac{d}{dt} \| g^{n+1}(t) \|_{E}^{2} &+ \| g^{n+1}(t) \|_{D}^{2} + \sum_{|\alpha| \leq N} \left( L_T \partial^\alpha g^{n+1}, \partial^\alpha g^{n+1} \right) \\
	  		&= \sum_{|\alpha| \leq N} \left( K_B \partial^\alpha g^{n}, \partial^\alpha g^{n+1} \right) \\
	  		&+ \sum_{|\alpha| \leq N} \left( \partial^\alpha \left[ \Gamma_{B}(g^{n}, g^{n}) + \Gamma_{T}(g^{n}, g^{n}) + T(g^{n}, g^{n}, g^{n}) \right]\!, \partial^\alpha g^{n+1} \right)\!.  
	  	\end{align*}
  	Utilizing the positivity of the operator \(L_T\) and the estimates from Corollary \ref{Estimates of LK}, we obtain  
  	\[
  	\left( L_T \partial^\alpha g^{n+1}, \partial^\alpha g^{n+1} \right) \geq 0,  
  	\]
  	and\[
  	  \sum_{|\alpha| \leq N} \left( K_B \partial^\alpha g^{n}, \partial^\alpha g^{n+1} \right) \lesssim \| g^{n} \|_{E} \| g^{n+1} \|_{E}.
  	\] 
  	Furthermore, by combining Lemma \ref{Estimates of GammaB}, Lemma \ref{Estimates of T}, and the Sobolev inequalities from Lemma \ref{Sobolev inequalities}, the regularity condition \( N \geq 2 \) allows us to derive the following bound:
  	\begin{align*}
  		\sum_{|\alpha| \leq N} & \left( \partial^\alpha \left[ \Gamma_{B}(g^{n}, g^{n}) + \Gamma_{T}(g^{n}, g^{n}) + T(g^{n}, g^{n}, g^{n}) \right]\!, \partial^\alpha g^{n+1} \right) \\ 
  		\lesssim &  \left( \| g^{n} \|_{E} \| g^{n} \|_{D} + \| g^{n} \|^{2}_{E} \| g^{n} \|_{D} \right) \| g^{n+1} \|_{D}.
  	\end{align*} 
  	These estimates collectively imply that
  	\begin{align*}
  		\frac{1}{2} \frac{d}{dt} & \| g^{n+1}(t) \|_{E}^{2} + \| g^{n+1}(t) \|_{D}^{2} \\
  		\lesssim&  \| g^{n} \|_{E} \| g^{n+1} \|_{E} + \left( \| g^{n} \|_{E} \| g^{n} \|_{D} + \| g^{n} \|^{2}_{E} \| g^{n} \|_{D} \right) \| g^{n+1} \|_{D} \\
  		\lesssim&  \varepsilon \left( \| g^{n+1} \|^{2}_{E} + \| g^{n+1} \|^{2}_{D} \right) + C_{\varepsilon}\|g^{n}\|_{E}^{2} + C_{\varepsilon} \left( \| g^{n} \|^{2}_{E} + \| g^{n} \|^{4}_{E} \right) \| g^{n} \|^{2}_{D}.
  	\end{align*} 
	  	Integrating the above inequality from 0 to \(t\) for arbitrary \(t \in [0, T^*]\), taking the supremum over \(t \in [0, T^*]\), and selecting sufficiently small \(\varepsilon > 0\), the induction hypothesis \eqref{induction gn} yields 
	  	$$\mathbb{E}_{T^{*}}(g^{n+1})\le \frac{4}{3}\|g_{0}\|_{E}^{2}+C(T^*+M_{0}+M_{0}^{2})\|g_{0}\|_{E}^{2}.$$
	  	 Finally, selecting smaller parameters \(T^* > 0\) and \(M_0 > 0\) when necessary such that   
	   $$C(T^*+M_{0}+M_{0}^{2})\le\frac{1}{3}.$$
	  	 Consequently,   
	  	$$\mathbb{E}_{T^{*}}(g^{n+1})\le 2\|g_{0}\|_{E}^{2}.$$
	  	 This completes the proof of the lemma. 
	  \end{proof}
	  Next, we define  \(h^{n+1} = g^{n+1} - g^{n}\). 
	  Then, from \eqref{iterative scheme g 1}, we deduce that \(h^{n+1}\) satisfies the following equation:
	  \begin{align*}
	  	\partial_t h^{n+1} &+ v \cdot \nabla_x h^{n+1} + \nu_{B}(v)h^{n+1} + L_{T}h^{n+1} \\  
	  	=& K_{B}h^{n} + \Gamma_{B}(h^{n}, g^{n}) +  \Gamma_{B}(g^{n-1}, h^{n})+\Gamma_{T}(h^{n}, g^{n})+ \Gamma_{T}(g^{n-1}, h^{n})\\
	  	& + T(h^{n}, g^{n}, g^{n})+ T(g^{n-1}, h^{n}, g^{n})+ T(g^{n-1}, g^{n-1}, h^{n}),
	  \end{align*} 
	  with initial data \(h^{n+1}(0) = 0\).
	  Using the uniform bounds established in Lemma \ref{uniform estimate gn} and employing a similar proof strategy, we deduce that for sufficiently small \(T^* > 0\) and an appropriately chosen \(M_0 > 0\), it can be shown that 
	  $$\mathbb{E}_{T^{*}}(h^{n+1})\le \frac{1}{2}\mathbb{E}_{T^{*}}(h^{n}).$$
	  Therefore,  the sequence \(\{g^n\}\) forms a Cauchy sequence in the space  
	  \[
	  L^{\infty}\big([0, T^*]; H^{N}_{x}L^{2}_{v}\big) \cap L^{2}\big([0, T^*]; H^{N}_{x}L^{2}_{\mathcal{D}}\big).
	  \]
	  Taking the limit as \(n \to \infty\) yields a solution \(g\) to equation \eqref{BT-g}, which satisfies the energy estimate 
	  \[
	  \mathbb{E}_{T^{*}}(g) \leq 2\|g_{0}\|_{E}^{2}.
	  \]
	  
	   Regarding uniqueness, suppose there exists another solution \(\tilde{g}\) to equation \eqref{BT-g} satisfying   
	  \[
	  \mathbb{E}_{T^{*}}(\tilde{g}) \leq 2M_{0}.
	  \]  
	   Define the difference \(h = g - \tilde{g}\). Then, employing a similar argument as above, it can be shown that   
	  \[
	  \mathbb{E}_{T^{*}}(h) \leq \frac{1}{2} \mathbb{E}_{T^{*}}(h),
	  \]  
	   which implies  \(\mathbb{E}_{T^{*}}(h) = 0\),  thus establishing uniqueness. 
	  
	 To demonstrate the continuity of the solution \(g\) with respect to time \(t\), we reformulate \eqref{BT-g} as the following system:   
	  \[
	  \begin{cases}  
	  	\partial_t g + v \cdot \nabla_x g + \nu_{B}(v)g + L_{T} g = K_{B}g + \Gamma_{B}(g, g) + \Gamma_{T}(g, g) + T(g, g, g), \\  
	  	g(0, x, v) = g_0(x, v).  
	  \end{cases}
	  \]  
	 Using the energy identity, justified by regularization, together with Corollary \ref{Estimates of LK} and the nonlinear estimates, we obtain
     \[
     \left|\frac{d}{dt}\|g(t)\|_E^2\right|
     \lesssim \left(1+\|g(t)\|_E+\|g(t)\|_E^2\right)\|g(t)\|_D^2.
     \]
	 Hence, for any  \(s, t \in [0, T^*]\), integrating the above inequality from \(\min\{s,t\}\) to \(\max\{s,t\}\) yields 
	 \[
	 \left|\, \| g(t) \|_{E}^{2} - \| g(s) \|_{E}^{2} \,\right| \lesssim \left(1 + \sqrt{M_0} + M_0\right) \int_{\min\{s,t\}}^{\max\{s,t\}} \| g(\tau) \|_{D}^{2}  d\tau \to 0, \quad \text{as } t \to s.
	 \]
	  
	  \noindent \textit{Step 2. Positivity.}
	  
	  The approximating sequence from Step 1 is inadequate for proving positivity.
	  Therefore, we construct a modified iterative scheme as follows: 
	     	\small
	     	\begin{align}\label{iterate f(n)}
	     		\left\{
	     		\begin{aligned}
	     			&\left(\partial_t  + v \cdot \nabla_x+R_{B}(f^{n})+R_{T}(f^{n},f^{n})\right)f^{n+1}=Q_{B,+}(f^{n},f^{n})+Q_{T,+}(f^{n},f^{n},f^{n}),\\[0.5mm] 
	     			&f^{n+1}(0,x,v)= f_0(x,v),\\
	     		\end{aligned}
	     		\right.
	     	\end{align}
     	\normalsize 
     	where $ Q_{B,+}, $ $ R_{B}, $  $ Q_{T,+}, $  $ R_{T} $ are defined in \eqref{eq:Qz1}, \eqref{eq:Qz2}, \eqref{eq:QT_plus} and \eqref{eq:QT_minus}. 
     	It follows directly from the iterative scheme \eqref{iterate f(n)} that if \(f_0 \geq 0\) and \(f^n \geq 0\), then
     	\begin{align*}
     		f^{n+1}(t,x,v) = & \, f_{0}(x - tv, v) \exp\left(-\int_{0}^{t} \left(R_{B}(f^{n}) + R_{T}(f^{n}, f^{n})\right)(s, x - (t - s)v, v) \, ds\right) \\
     		& + \int_{0}^{t} \left(Q_{B,+}(f^{n}, f^{n}) + Q_{T,+}(f^{n}, f^{n}, f^{n})\right)(s, x - (t - s)v, v) \\
     		& \quad \times \exp\left(-\int_{s}^{t} \left(R_{B}(f^{n}) + R_{T}(f^{n}, f^{n})\right)(\tau, x - (t - \tau)v, v) \, d\tau\right) ds \\
     		\ge & \, 0.
     	\end{align*}
       Consequently, the iterative scheme preserves positivity.
     	In the iterative scheme \eqref{iterate f(n)}, setting \( f^{n+1}(t,x,v) = \mu(v) + \sqrt{\mu(v)} \bar{g}^{n+1}(t,x,v) \), we can readily obtain the equation for \( \bar{g}^{n+1} \) as 
     	\begin{align}\label{iteration bargn}
     		\begin{aligned}
     			\partial_t \bar{g}^{n+1}&  + v \cdot \nabla_x \bar{g}^{n+1}+ (\nu_{B}(v)+\nu_{T}(v))\bar{g}^{n+1} \\[1mm]
     			=& (K_{B}+K_{T})\bar{g}^{n}+\Gamma_{B,+}(\bar{g}^{n},\bar{g}^{n})-\Gamma_{B,-}(\bar{g}^{n+1},\bar{g}^{n})\\[1mm]
     			&+T(\sqrt{\mu},\bar{g}^{n},\bar{g}^{n})+T_{+}(\bar{g}^{n},\sqrt{\mu},\bar{g}^{n})-T_{-}(\bar{g}^{n+1},\sqrt{\mu},\bar{g}^{n})\\[1mm]
     			&+T_{+}(\bar{g}^{n},\bar{g}^{n},\sqrt{\mu})-T_{-}(\bar{g}^{n+1},\bar{g}^{n},\sqrt{\mu})\\[1mm]
     			&+T_{+}(\bar{g}^{n},\bar{g}^{n},\bar{g}^{n})-T_{-}(\bar{g}^{n+1},\bar{g}^{n},\bar{g}^{n}),
     		\end{aligned} 
     	\end{align}  
     	with $\bar{g}^{n+1}(0) = g_0,\quad \bar{g}^{0}=0.$
     	
     	\noindent{Case (i):} $\gamma_{3}\in(-3,0].$
     	
     	For the case of ternary collisions with soft potentials, we observe from Corollary \ref{Estimates of LK} that  
     	\[
     	|\langle K_B g, h \rangle| \lesssim\|g\|_{L^2_{v}}\|h\|_{L^2_{v}},
     	\]
     	$$|\langle K_{T}g,h\rangle|\lesssim  \|g\|_{L^{2}_{\gamma_{3}}}\|h\|_{L^{2}_{\gamma_{3}}}\lesssim\|g\|_{L^2_{v}}\|h\|_{L^2_{v}}.$$    
     	Incorporating the coercivity estimate
     	\[
     	  \langle(\nu_{B}(v)+\nu_{T}(v))g,g\rangle \geq \langle \nu_{B}(v)g,g\rangle = \|g\|^{2}_{L^{2}_{\mathcal{D}}},
     	  \]
     	  together with the nonlinear estimates from Lemmas \ref{Estimates of GammaB} and \ref{Estimates of T} and adapting the proof strategy of Lemma \ref{uniform estimate gn} from Step 1, we show that there exist sufficiently small parameters \(T^*,\) \(M_1 \in (0,1)\) such that if
     	  \[
     	  \|g_0\|^{2}_{H^N_x L^2_v} \leq M_{1},
     	  \]
     	  then the following uniform bound holds:
     	  \[
     	  \sup_{n \geq 0} \mathbb{E}_{T^*}(\bar{g}^n) \leq 2 \|g_0\|_E^2.
     	  \]
     	Taking the limit as \( n \to \infty \) in the iterative scheme \eqref{iteration bargn} establishes the existence of a local-in-time solution \(\bar{g}\) to equation \eqref{BT-g}. 
     	This solution satisfies   \(\mu + \sqrt{\mu} \bar{g} \geq 0\) and the energy bound
     	\[
     	\mathbb{E}_{T^*}(\bar{g}) \leq 2 \|g_0\|_E^2.
     	\]
     	Thus, by selecting \(\|g_0\|_E^2 \leq M_0 \leq M_1\), it follows from the uniqueness result established in Step 1 that the local solution \(g\) to equation \eqref{BT-g} constructed in Step 1 also satisfies  \(\mu + \sqrt{\mu} g \geq 0\).

     	\noindent{Case (ii):} $\gamma_{3}\in(0,1].$
     	
     	For the case of ternary collisions with hard potentials, we note that due to the growth of \(\langle v\rangle^{\gamma_{3}}\), the operator \(K_{T}\) may not be bounded in \(L^2_v\) according to Corollary \ref{Estimates of LK}. To address this, we introduce a truncation of the ternary collision kernel \(B_3\): for any positive integer \(m\), define  
     	\[
     	B_m := \min\{B_3, m\}.
     	\]  
     	Replacing the original kernel \(B_3\) in equation \eqref{BT-g} with this truncated version, we obtain the following modified equation:  
     	\small
     	\begin{equation}\label{BT-gm}
     		\begin{cases}  
     			\partial_t \bar{g}^{m} + v \cdot \nabla_x \bar{g}^{m} + L_{B}\bar{g}^{m} + L_{T,m}\bar{g}^{m} = \Gamma_{B}(\bar{g}^{m}, \bar{g}^{m}) + \Gamma_{T,m}(\bar{g}^{m}, \bar{g}^{m}) + T_{m}(\bar{g}^{m}, \bar{g}^{m}, \bar{g}^{m}), \\  
     			\bar{g}^{m}(0, x, v) = g_0(x, v).  
     		\end{cases}  
     	\end{equation}  
     \normalsize
     	Moreover, the following estimates hold for the truncated operators:  
     	\[
     	0 \leq \nu_{T,m}(v) \lesssim \min\left\{\langle v \rangle^{\gamma_3}, m\right\},
     	\]  
     	and  
     	\[
     	\left| \langle K_{T,m}g, h \rangle \right| \lesssim m \|g\|_{L^2_v} \|h\|_{L^2_v}.
     	\]
     	Evidently, the upper bound estimates for the nonlinear terms stated in Lemma \ref{Estimates of T} and Corollary \ref{Estimates of GammaT} remain valid for \(T_m\) and \(\Gamma_{T,m}\). We now employ the positivity-preserving iterative scheme associated with \eqref{iteration bargn} to solve \eqref{BT-gm}. 
     	Adapting the method used in Case (i), we deduce that for any fixed positive integer \(m\), there exist parameters \(T^m > 0\) and \(M_1 \in (0,1)\) such that, under the condition \(\|g_0\|^{2}_{H^N_x L^2_v} \leq M_1\), there exists a solution \(\bar{g}^{m}\) to equation \eqref{BT-gm} on $[0,T^{m}]$ satisfying 
     	$$\mu + \sqrt{\mu}\bar{g}^{m} \geq 0$$
     	and
     	\[
     	\mathbb{E}_{T^m}(\bar{g}^{m}) \leq 2 \|g_0\|_E^2.
     	\]
     	On the other hand, we reformulate \eqref{BT-gm} as follows:
     	\[
     	\begin{cases}
     		\partial_t \bar{g}^{m} + v \cdot \nabla_x \bar{g}^{m} + \nu_{B}\bar{g}^{m} + L_{T,m}\bar{g}^{m} = K_{B}\bar{g}^{m} + \Gamma_{B}(\bar{g}^{m}, \bar{g}^{m})\\
     		\quad\quad\quad\quad\quad\quad\quad\quad\quad\quad\quad\quad\quad\quad\quad\quad\quad + \Gamma_{T,m}(\bar{g}^{m}, \bar{g}^{m}) + T_{m}(\bar{g}^{m}, \bar{g}^{m}, \bar{g}^{m}), \\
     		\bar{g}^{m}(0, x, v) = g_0(x, v).
     	\end{cases}
     	\]
     	Utilizing the  positivity of \(L_{T,m}\)  and  following the approach from Step 1, we can prove that there exist constants \(T^* > 0\) and \(M_0 > 0\),  both independent of \(m\), such that under the condition \(\|g_0\|^{2}_{H^N_x L^2_v} \leq M_0\), the solution \(\bar{g}^{m}\) satisfies the following a priori  estimate:
     	\begin{equation}\label{a priori bargm}
     		\mathbb{E}_{T^*}(\bar{g}^{m}) \leq 2 \|g_0\|_E^2.
     	\end{equation} 
     	In other words, for every \(m\), the existence interval of \(\bar{g}^{m}\) can be extended to \([0, T^*]\), with the solution satisfying \(\mu + \sqrt{\mu}\bar{g}^{m} \geq 0\) and the following uniform estimate:
     	\[
     	\sup_m \, \mathbb{E}_{T^*}(\bar{g}^{m}) \leq 2 \|g_0\|_E^2.
     	\]
     	Specifically, assume \(\|g_0\|^{2}_{H^N_x L^2_v} \leq M_0 \leq \frac{M_1}{2}\). If \(T^{m} < T^*\), then from \eqref{a priori bargm} it follows that \(\|\bar{g}^{m}(T^{m})\|^{2}_{H^N_x L^2_v} \leq 2\|g_0\|^{2}_{H^N_x L^2_v} \leq M_1\). We can then use \(\bar{g}^{m}(T^{m})\) as initial data and extend the solution to the interval \([0, 2T^{m}]\) via the positivity-preserving iterative scheme \eqref{iteration bargn}. Again, from \eqref{a priori bargm}, we obtain \(\|\bar{g}^{m}(2T^{m})\|^{2}_{H^N_x L^2_v} \leq 2\|g_0\|^{2}_{H^N_x L^2_v} \leq M_1\). By repeating this process a finite number of times, the solution \(\bar{g}^{m}\) can necessarily be extended to the full interval \([0, T^*]\). 
     	Therefore, taking the limit \( m \to \infty \) in \eqref{BT-gm} yields a solution \(\bar{g}\) to equation \eqref{BT-g}, which satisfies the nonnegativity condition \(\mu + \sqrt{\mu}\bar{g} \geq 0\) and the energy bound  
     	\[
     	\mathbb{E}_{T^*}(\bar{g}) \leq 2 \|g_0\|_E^2.
     	\]  
     	By uniqueness (Step 1), the solution \(g\) constructed therein also satisfies \(\mu + \sqrt{\mu} g \geq 0\).

	     \end{proof}

		\paragraph{4.2. Macroscopic Estimate.}\label{macroscopic estimate}
		
		By Proposition \ref{prop:L_properties}, the macro-micro decomposition for the equation \eqref{BT-g} is formulated as:  
		\begin{align}\label{macro-micro decomposition}
			\begin{aligned} 
				g = & \mathbf{P}g + \{\mathbf{I} - \mathbf{P}\}g\\ 
				=&:g_{1}+g_{2} \\ 
				=& \sqrt{\mu}(a + \mathbf{b} \cdot v + c|v|^2)+g_{2},
			\end{aligned}   
		\end{align}   
		where $ \mathbf{b}=(b_{1},b_{2},b_{3}) $ and \( \mathbf{P} \) denotes the orthogonal projection onto \( \operatorname{Ker}(L) .\)
		The derivation of the macroscopic equations can be obtained by similar steps as in \cite{MR2420519}. For the sake of completeness, we will restate it here.
		By substituting $g= g_{1} + g_{2}$ into equation \eqref{BT-g}, we express the hydrodynamic component $g_{1}$ in terms of the microscopic component $g_{2}$ and the nonlinear terms $\Gamma_{B},$ $\Gamma_{T}$ and $T$.
		\[
		(\partial_t + v\cdot \nabla_x)g_{1} = -\partial_tm+l+G ,
		\] 
		where
		$$m:=g_{2},\quad l:=- (v \cdot \nabla_x+L)g_{2},\quad G:= \Gamma_{B}(g,g)+ \Gamma_{T}(g,g) + T(g,g,g). $$
		Subsequently, based on \eqref{macro-micro decomposition}, we expand the left-hand side of the equation as:
		\small
		\[
		\left( \partial_t a + \sum_{i}\left((\partial_{t}b_{i}+ \partial_{x_i}a) v_{i} +   \partial_{x_i}c v_i|v|^{2} + (\partial_t c+\partial_{x_i} b_i) |v_{i}|^{2}\right) +  \sum_{i<j} (\partial_{x_i} b_j +\partial_{x_j} b_i)v_i v_j \right)\sqrt{\mu}.
		\]  
		\normalsize
		This is a linear combination of the 13 basis functions,  
		\[
		\left\{e_{k}\right\}_{k=1}^{13}:=\left\{
		\sqrt{\mu}, \  {v_i}\sqrt{\mu}, \  |v_{i}|^{2}\sqrt{\mu }, \   {v_iv_j} \sqrt{\mu}, \  {v_i}|v|^{2}\sqrt{\mu}
		\right\},
		\] 
		for \( 1 \leq i, j \leq 3 \).
		By expanding the right-hand side with respect to the same basis  and comparing the corresponding coefficients on both sides, we derive the following macroscopic equations for $a$, $b_i$, and $c:$ 
		\begin{equation}\label{Macro equation}
			\begin{aligned}
				\partial_t a &=-\partial_{t} m_{a}+ l_a + G_a, \\
				\partial_{t}b_{i}+ \partial_{x_i}a &=-\partial_{t} m_{abi}+ l_{abi} + G_{abi}, \\
				\partial_t c+\partial_{x_i} b_i &=-\partial_{t} m_{bc}+ l_{bc} + G_{bc}, \\
				\partial_{x_i} b_j +\partial_{x_j} b_i &=-\partial_{t} m_{ij}+ l_{ij} + G_{ij}, \\
				\partial_{x_i}c &=-\partial_{t} m_{c}+ l_c + G_c,
			\end{aligned}
		\end{equation} 
		where the indices are taken from the set \( D = \{ a, abi, bc, ij, c \mid 1 \leq i \leq j \leq 3 \} \)  and \( m_\lambda \), \( l_\lambda \), and \( G_\lambda \) for \( \lambda \in D \) represent the coefficients of $ m, $ $ l, $ and $ G $ with respect to the basis \( \{ e_k \}_{k=1}^{13} \), respectively.
		It is evident that, for any multi-index $\alpha$ and real number $m$, the following estimate holds:
			\begin{equation}\label{m-l-estimate}
			\sum_{\lambda\in D}\|\partial^{\alpha}[m_{\lambda},l_{\lambda}]\|^{2}_{L^{2}_{x}}\lesssim\sum_{k=1}^{13}\|\langle\partial^{\alpha}[m,l],e_{k}\rangle\|^{2}_{L^{2}_{x}}\lesssim\sum_{|\alpha'|\le|\alpha|+1}\|\partial^{\alpha'}g_{2}\|^{2}_{{L^{2}_{x}L^{2}_{m}}}.
		\end{equation} 
		Moreover, following the methodology outlined in \cite{MR2420519}, and utilizing conservation laws \eqref{conservation laws g}, we deduce that the variables $a$, $\mathbf{b}$, and $c$ satisfy the following macroscopic conservation laws.
		\begin{align}\label{local conservation}
			\left\{
			\begin{aligned}
				&\partial_t a 
				 - \frac{1}{2} \nabla_x \cdot \langle |v|^2 v \sqrt{\mu}, g_2 \rangle = 0, \\
				&\partial_t \mathbf{b}
				 + \nabla_x (a + 5c) + \nabla_x \cdot \langle v \otimes v \sqrt{\mu}, g_2 \rangle = 0,\\
				&\partial_t c 
				 + \frac{1}{3} \nabla_x \cdot \mathbf{b} + \frac{1}{6} \nabla_x \cdot \langle |v|^2 v \sqrt{\mu}, g_2 \rangle = 0.
			\end{aligned}
			\right.
		\end{align} 
		 
		For $N\ge1,$ we hereby define the dissipation rate, denoted as $ \mathcal{D}_{N} $, as follows:
		\begin{equation}\label{dissipation rate}
			\mathcal{D}_{N}(g(t)):=\sum_{0<|\alpha|\le N}^{}\|\partial^{\alpha}[a,\mathbf{b},c]\|^{2}_{L^{2}_{x}} +\sum_{0\le|\alpha|\le N}^{}\|\partial^{\alpha}g_{2}(t)\|^{2}_{L^{2}_{x}L^{2}_{\mathcal{D}}}. 
		\end{equation} 
	  It is evident that this dissipation rate is equivalent to the squared $N$-th order Sobolev norm $\| \cdot \|_{H^{N}_x L^2_{\mathcal{D}}}^2$, excluding the zeroth-order macroscopic component $g_1$, that is,
	  \[\mathcal{D}_{N}(g(t))\sim \| g_{2}(t)\|^{2}_{L^{2}_{x}L^{2}_{\mathcal{D}}}+\sum_{0<|\alpha|\le N}^{}\|\partial^{\alpha}g(t)\|^{2}_{L^{2}_{x}L^{2}_{\mathcal{D}}}.\]
	   
		We now present the principal estimates concerning the macroscopic dissipation rate in the following lemma. 
\begin{lemma}\label{Macro estimate}
			Let \( N\ge2 \), \( \gamma_2 \in (-\frac{3}{2}, 1] \) and \( \gamma_3 \in (-3, 1] \).  There exists a positive constant \( C_1 \) such that
			\small
			\begin{align}\label{Macro estimate 1}
				\begin{aligned}
					\frac{d\mathcal{I}(t)}{dt}+&\|\nabla_x[a,\mathbf{b},c](t)\|^{2}_{H^{N-1}_{x}} \\[2mm]
					\le &C_{1}\left( \| g_{2}(t)\|^{2}_{H^{N}_{x}L^{2}_{\mathcal{D}}}+(\|g(t)\|^{2}_{H^{N}_{x}L^{2}_{v}}+\|g(t)\|^{4}_{H^{N}_{x}L^{2}_{v}})\mathcal{D}_{N}(g(t))\right),
				\end{aligned} 
			\end{align} 
		\normalsize
			where \( \mathcal{I}(t) \) is a temporal interaction energy functional that satisfies 
			$$|\mathcal{I}(t)|\le C_{0}\|g(t)\|^{2}_{H^{N}_{x}L^{2}_{v}}$$
			with some constant  $  C_{0}>0. $
		\end{lemma} 
		\begin{proof}
			By employing a methodology analogous to that developed in the proof of Theorem 3.1 in \cite{MR2420519}, and incorporating equation \eqref{m-l-estimate}, we derive that for all integers $N \geq 1$, there exists a temporal interaction energy functional $\mathcal{I}(t)$ satisfying the bound 
			$$
			|\mathcal{I}(t)| \lesssim \|g(t)\|^{2}_{H^{N}_{x}L^{2}_{v}},
			$$ 
			such that the macroscopic dissipation rate adheres to the inequality 
			$$
			\frac{d\mathcal{I}(t)}{dt} + \|\nabla_x[a,\mathbf{b},c](t)\|^{2}_{H^{N-1}_{x}} \lesssim  \| g_{2}(t)\|^{2}_{H^{N}_{x}L^{2}_{\mathcal{D}}} + \sum_{k=1}^{13}\|\langle G,e_{k}\rangle\|^{2}_{H^{N}_{x}}.
			$$
			For each $e_k$ with $k = 1, \ldots, 13$, we estimate $\| \langle G, e_k \rangle \|^2_{H^N_x}$. For any multi-index $\alpha$, we have
			\begin{equation}\label{partial-GB}
				\partial^{\alpha}\Gamma_{B}(g,g)=\sum_{ \alpha'+\alpha''=\alpha}^{}C^{\alpha}_{\alpha'\alpha''} \Gamma_{B}(\partial^{\alpha'}g,\partial^{\alpha''}g),
			\end{equation}
			\begin{equation}\label{partial-GT}
				\partial^{\alpha}\Gamma_{T}(g,g)=\sum_{ \alpha'+\alpha''=\alpha}^{}C^{\alpha}_{\alpha'\alpha''} \Gamma_{T}(\partial^{\alpha'}g,\partial^{\alpha''}g),
			\end{equation}
			and
			\begin{equation}\label{partial-T}
				\partial^{\alpha}T(g,g,g)=\sum_{ \alpha'+\alpha''+\alpha'''=\alpha}^{}C^{\alpha}_{\alpha'\alpha''\alpha'''} T(\partial^{\alpha'}g,\partial^{\alpha''}g,\partial^{\alpha'''}g).
			\end{equation} 
			Therefore, by Corollary \ref{Estimates of Gamma T mu}, we deduce that 
			\begin{align}\label{partial-GB-GT ek}
				\begin{aligned}
					& |\langle\partial^{\alpha}\Gamma_{B}(g,g),e_{k} \rangle|+ |\langle\partial^{\alpha}\Gamma_{T}(g,g),e_{k} \rangle|\\
					\lesssim&\sum_{ |\alpha'|+|\alpha''|=|\alpha|}^{}\|\mu^{\delta} \partial^{\alpha'}g\|_{L^{2}_{v}} \|\mu^{\delta} \partial^{\alpha''}g\|_{L^{2}_{v}}\\
					\lesssim&\sum_{ |\alpha'|+|\alpha''|=|\alpha|}^{}\min\{\|\partial^{\alpha'}g\|_{L^{2}_{v}} \| \partial^{\alpha''}g\|_{L^{2}_{\mathcal{D}}},\|\partial^{\alpha'}g\|_{L^{2}_{\mathcal{D}}} \| \partial^{\alpha''}g\|_{L^{2}_{v}}\},
				\end{aligned} 
			\end{align}
			and
			\small
			\begin{align}\label{partial-T ek}
				\begin{aligned}
					 & |\langle\partial^{\alpha}T(g,g,g),e_{k} \rangle|\\
					 \lesssim&\sum_{ |\alpha'|+|\alpha''|+|\alpha'''|=|\alpha|}^{}\|\mu^{\delta} \partial^{\alpha'}g\|_{L^{2}_{v}} \|\mu^{\delta} \partial^{\alpha''}g\|_{L^{2}_{v}}\|\mu^{\delta} \partial^{\alpha'''}g\|_{L^{2}_{v}}\\
					 \lesssim&\sum_{ |\alpha'|+|\alpha''|+|\alpha'''|=|\alpha|}^{}\min\{\|\partial^{\alpha'}g\|_{L^{2}_{v}} \| \partial^{\alpha''}g\|_{L^{2}_{v}}\| \partial^{\alpha'''}g\|_{L^{2}_{\mathcal{D}}},\\
					 &\quad\quad\quad\quad \|\partial^{\alpha'}g\|_{L^{2}_{v}} \| \partial^{\alpha''}g\|_{L^{2}_{\mathcal{D}}}\| \partial^{\alpha'''}g\|_{L^{2}_{v}},\|\partial^{\alpha'}g\|_{L^{2}_{\mathcal{D}}} \| \partial^{\alpha''}g\|_{L^{2}_{v}}\| \partial^{\alpha'''}g\|_{L^{2}_{v}}\}.
				\end{aligned}
			\end{align}
		\normalsize
			For all multi-indices $\alpha$ with $|\alpha| \leq N$ and $ N\ge2 $, we estimate the squared $L^2_x$-norms of \eqref{partial-GB-GT ek} and \eqref{partial-T ek}, denoted by $\| \cdot \|_{L^2_x}^2$. 
			Firstly, from \eqref{partial-GB-GT ek}, we observe that 
			\begin{align*}
				& |\langle\partial^{\alpha}\Gamma_{B}(g,g),e_{k} \rangle|+ |\langle\partial^{\alpha}\Gamma_{T}(g,g),e_{k} \rangle|\\ 
				\lesssim&\sum_{ |\alpha'|+|\alpha''|\le N}^{}\min\{\|\partial^{\alpha'}g\|_{L^{2}_{v}} \| \partial^{\alpha''}g\|_{L^{2}_{\mathcal{D}}},\|\partial^{\alpha'}g\|_{L^{2}_{\mathcal{D}}} \| \partial^{\alpha''}g\|_{L^{2}_{v}}\}.
			\end{align*} 
			We consider the following two cases.
			
			 \noindent{Case (i):}  $|\alpha'| = N$ or $|\alpha''| = N$.
			
			Without loss of generality, we assume $|\alpha'| = N$, which implies $|\alpha''| = 0$. 
			Moreover, we estimate $\partial^{\alpha'}g$ using the $L^2_{x,v}$-norm, as it represents a top-order term for which the Sobolev embedding is not applicable.
			Hence, by applying the Sobolev inequality presented in Lemma \ref{Sobolev inequalities}, we derive that
			\begin{align*}
				  \left\|\|\partial^{\alpha'}g\|_{L^{2}_{v}} \|g\|_{L^{2}_{\mathcal{D}}}\right\|^{2}_{L^2_{x}} 
				\lesssim&\|\partial^{\alpha'}g\|^{2}_{L^{2}_{x}L^{2}_{v}} \|g\|^{2}_{L^{\infty}_{x}L^{2}_{\mathcal{D}}}\\
				\lesssim&\|g\|^{2}_{H^{N}_{x}L^{2}_{v}} \|\nabla_xg\|^{2}_{H^{1}_{x}L^{2}_{\mathcal{D}}}\\
				\lesssim&\|g(t)\|^{2}_{H^{N}_{x}L^{2}_{v}}\mathcal{D}_{N}(g(t)).
			\end{align*} 
			
			\noindent{Case (ii):}  $|\alpha'|,|\alpha''|\le N-1$.
			
			In this case, we continue to apply Lemma \ref{Sobolev inequalities} to conclude that
			\begin{align*}
				\left\|\|\partial^{\alpha'}g\|_{L^{2}_{v}} \|\partial^{\alpha''}g\|_{L^{2}_{\mathcal{D}}}\right\|^{2}_{L^2_{x}} 
				\lesssim&\|\partial^{\alpha'}g\|^{2}_{L^{3}_{x}L^{2}_{v}} \|\partial^{\alpha''}g\|^{2}_{L^{6}_{x}L^{2}_{\mathcal{D}}}\\
				\lesssim&\|\partial^{\alpha'}g\|^{2}_{H^{1}_{x}L^{2}_{v}} \|\nabla_x \partial^{\alpha''}g\|^{2}_{L^{2}_{x}L^{2}_{\mathcal{D}}}\\
				\lesssim&\|g(t)\|^{2}_{H^{N}_{x}L^{2}_{v}}\mathcal{D}_{N}(g(t)).
			\end{align*}

			Next, we proceed to estimate the $ L^{2}_{x} $ norm of \eqref{partial-T ek}. 
			Obviously, we note that
			\begin{align*}
				& |\langle\partial^{\alpha}T(g,g,g),e_{k} \rangle|\\ 
				\lesssim&\sum_{ |\alpha'|+|\alpha''|+|\alpha'''|\le N}^{}\min\{\|\partial^{\alpha'}g\|_{L^{2}_{v}} \| \partial^{\alpha''}g\|_{L^{2}_{v}}\| \partial^{\alpha'''}g\|_{L^{2}_{\mathcal{D}}},\\
				&\quad\quad\quad\quad \|\partial^{\alpha'}g\|_{L^{2}_{v}} \| \partial^{\alpha''}g\|_{L^{2}_{\mathcal{D}}}\| \partial^{\alpha'''}g\|_{L^{2}_{v}},\|\partial^{\alpha'}g\|_{L^{2}_{\mathcal{D}}} \| \partial^{\alpha''}g\|_{L^{2}_{v}}\| \partial^{\alpha'''}g\|_{L^{2}_{v}}\}.
			\end{align*}  
			Similarly, we consider the following two cases.
			
			\noindent{Case (i):}  $|\alpha'| = N$ or $|\alpha''| = N$ or $|\alpha'''| = N$.
			
			Without loss of generality, we assume $|\alpha'| = N$, which implies $|\alpha''| = |\alpha'''| = 0$. 
			In this scenario, we estimate $\partial^{\alpha'} g$ using the $L^2_{x,v}$-norm. 
			By applying  Lemma \ref{Sobolev inequalities}, we get
			\begin{align*}
				\left\|\|\partial^{\alpha'}g\|_{L^{2}_{v}} \| g\|_{L^{2}_{v}}\|g\|_{L^{2}_{\mathcal{D}}}\right\|^{2}_{L^2_{x}} 
				\lesssim&\|\partial^{\alpha'}g\|^{2}_{L^{2}_{x}L^{2}_{v}} \|g\|^{2}_{L^{\infty}_{x}L^{2}_{v}} \|g\|^{2}_{L^{\infty}_{x}L^{2}_{\mathcal{D}}}\\
				\lesssim&\|g\|^{2}_{H^{N}_{x}L^{2}_{v}} \|g\|^{2}_{H^{2}_{x}L^{2}_{v}}\|\nabla_xg\|^{2}_{H^{1}_{x}L^{2}_{\mathcal{D}}}\\
				\lesssim&\|g(t)\|^{4}_{H^{N}_{x}L^{2}_{v}}\mathcal{D}_{N}(g(t)).
			\end{align*} 
			
			\noindent{Case (ii):}  $|\alpha'|,|\alpha''|, |\alpha'''|\le N-1$. 
			
			Since \( N \geq 2 \), we observe that at least one of \( |\alpha'| \), \( |\alpha''| \), or \( |\alpha'''| \) must satisfy \( |\cdot| \leq N - 2 \). Otherwise, if \( |\alpha'| = |\alpha''| = |\alpha'''| = N - 1 \), then it would follow that  
			\[
			3N - 3 = |\alpha'| + |\alpha''| + |\alpha'''|  \leq N,
			\]  
			which contradicts the assumption that \( N \geq 2 \). 
			Without loss of generality, let us assume \( |\alpha'| \leq N - 2 \). Following analogous reasoning, we then deduce that
			\begin{align*}
				\left\|\|\partial^{\alpha'}g\|_{L^{2}_{v}} \| \partial^{\alpha''}g\|_{L^{2}_{v}}\| \partial^{\alpha'''}g\|_{L^{2}_{\mathcal{D}}}\right\|^{2}_{L^2_{x}} 
				\lesssim&\|\partial^{\alpha'}g\|^{2}_{L^{\infty}_{x}L^{2}_{v}} \|\partial^{\alpha''}g\|^{2}_{L^{3}_{x}L^{2}_{v}} \|\partial^{\alpha'''}g\|^{2}_{L^{6}_{x}L^{2}_{\mathcal{D}}}\\
				\lesssim&\|\partial^{\alpha'}g\|^{2}_{H^{2}_{x}L^{2}_{v}} \|\partial^{\alpha''}g\|^{2}_{H^{1}_{x}L^{2}_{v}}\|\nabla_x\partial^{\alpha'''}g\|^{2}_{L^{2}_{x}L^{2}_{\mathcal{D}}}\\
				\lesssim&\|g(t)\|^{4}_{H^{N}_{x}L^{2}_{v}}\mathcal{D}_{N}(g(t)).
			\end{align*} 
			
			Combining the above arguments, we conclude that for \( N\ge2 \),  
			\[  
			 \sum_{k=1}^{13}\|\langle G,e_{k}\rangle\|^{2}_{H^{N}_{x}}\lesssim (\|g(t)\|^{2}_{H^{N}_{x}L^{2}_{v}}+\|g(t)\|^{4}_{H^{N}_{x}L^{2}_{v}})\mathcal{D}_{N}(g(t)).
			\]  
			Therefore, the lemma is proved.

		\end{proof}

		\paragraph{4.3. Proof of Theorem \ref{main_theorem}.}
		
		Theorem \ref{local existence} indicates that the local solution can be extended to a global one via a bootstrap argument, provided that a globally sufficiently small estimate for the energy functional \(\|g(t)\|_{H^N_x L^2_v}\) is established. However, unlike the dissipation rate \(\|g(t)\|_{H^N_x L^2_\mathcal{D}}\) in Theorem \ref{local existence}, the dissipation \(\mathcal{D}_{N}(g(t))\) in \eqref{dissipation rate}—used to close the global estimates—lacks derivatives of order zero for the macroscopic components. This necessitates a more delicate analysis when estimating the nonlinear terms with this dissipation.

		We now proceed to derive a priori estimates for the $L^2$ energy functional. Applying the spatial derivative operator $\partial^\alpha$ to equation \eqref{BT-g}, multiplying both sides by $\partial^\alpha g$, and integrating over the variables $v$ and $x$, we obtain, upon summing over all multi-indices $\alpha$ with $|\alpha| \leq N$ and \(N\ge2\), the following identity: 
		\small
		\begin{equation}\label{energy identity}
			\frac{1}{2} \frac{d}{dt} \| g(t) \|^2_{H^N_x L^2_v } + \sum_{|\alpha| \leq N} \left( L \partial^\alpha g, \partial^\alpha g \right) = \sum_{|\alpha| \leq N} \left( \partial^\alpha \left( \Gamma_{B}(g, g)+\Gamma_{T}(g, g) + T(g, g, g) \right), \partial^\alpha g \right).
		\end{equation}
	   \normalsize 
	   Based on the methodologies and procedures outlined in references \cite{MR2095473,MR2420519}  concerning the binary Boltzmann equation, we similarly deduce that 
		\begin{equation}\label{Binary nonlinear estimate}
			\sum_{|\alpha| \leq N} \left|\left( \partial^\alpha \Gamma_{B}(g, g), \partial^\alpha g \right)\right|\lesssim \|g(t)\|_{H^{N}_{x}L^{2}_{v}} \mathcal{D}_{N}(g(t)).
		\end{equation}
		Therefore,   we proceed to estimate  
		  \[\sum_{|\alpha| \leq N}\left( \partial^\alpha \left( \Gamma_T(g, g) + T(g, g, g) \right), \partial^\alpha g \right). \] 
		On the one hand, we analyze the case \( |\alpha| = 0 \). 
		By Propositions \ref{prop:LT_properties} and \ref{prop:conservation}, we observe that
		\begin{align*}
			&\left\langle\Gamma_{{T}}(g,g)+ T(g,g,g),g_{1}\right\rangle\\
			=&\left\langle L_{T}g,g_{1}\right\rangle+\left\langle \mu^{-\frac{1}{2}} Q_{T}(f,f,f),g_{1}\right\rangle\\
			=&\left\langle g,L_{T}g_{1}\right\rangle+\left\langle  Q_{T}(f,f,f), a + \mathbf{b} \cdot v + c|v|^2\right\rangle\\
			=&0,
		\end{align*}
		from which we conclude that
		\begin{align*}
			&|\left\langle\Gamma_{{T}}(g,g)+ T(g,g,g),g\right\rangle|\\[1mm]
			\lesssim&|\left\langle\Gamma_{{T}}(g,g),g_{2}\right\rangle|+|\left\langle T(g,g,g),g_{2}\right\rangle|\\[1mm]
			\lesssim&\|g\|_{L^{2}_{v}}(1+\|g\|_{L^{2}_{v}})\|g\|_{L^{2}_{\mathcal{D}}}\|g_{2}\|_{L^{2}_{\mathcal{D}}} 
		\end{align*}
		by applying Lemma \ref{Estimates of T} and Corollary \ref{Estimates of GammaT}.
		Thus, we have
		\begin{align*} 
			&\left|\left(\Gamma_{{T}}(g,g)+T(g,g,g),g \right)\right|\\[1mm]
			\lesssim&\int_{\mathbb{R}^3}\left|\left\langle\Gamma_{{T}}(g,g)+T(g,g,g),g\right\rangle \right|dx\\[1mm]
			\lesssim&\int_{\mathbb{R}^3}\|g\|_{L^{2}_{v}}\|g\|_{L^{2}_{\mathcal{D}}}\|g_{2}\|_{L^{2}_{\mathcal{D}}} dx+\int_{\mathbb{R}^3}\|g\|_{L^{2}_{v}}\|g\|_{L^{2}_{v}}\|g\|_{L^{2}_{\mathcal{D}}}\|g_{2}\|_{L^{2}_{\mathcal{D}}} dx\\
			\lesssim&\|g\|_{L^{3}_{x}L^{2}_{v}}\|g\|_{L^{6}_{x}L^{2}_{\mathcal{D}}}\|g_{2}\|_{L^{2}_{x}L^{2}_{\mathcal{D}}}+\|g\|_{L^{\infty}_{x}L^{2}_{v}}\|g\|_{L^{3}_{x}L^{2}_{v}}\|g\|_{L^{6}_{x}L^{2}_{\mathcal{D}}}\|g_{2}\|_{L^{2}_{x}L^{2}_{\mathcal{D}}}\\[1mm]
			\lesssim&\|g\|_{H^{1}_{x}L^{2}_{v}}\|\nabla_xg\|_{L^{2}_{x}L^{2}_{\mathcal{D}}}\|g_{2}\|_{L^{2}_{x}L^{2}_{\mathcal{D}}}+\|g\|_{H^{2}_{x}L^{2}_{v}}\|g\|_{H^{1}_{x}L^{2}_{v}}\|\nabla_xg\|_{L^{2}_{x}L^{2}_{\mathcal{D}}}\|g_{2}\|_{L^{2}_{x}L^{2}_{\mathcal{D}}}\\[1mm]
			\lesssim&(\|g(t)\|_{H^{N}_{x}L^{2}_{v}}+\|g(t)\|^{2}_{H^{N}_{x}L^{2}_{v}})\mathcal{D}_{N}(g(t)).
		\end{align*}
		On the other hand, we estimate the following sum 
		\[
		\sum_{1 \leq |\alpha| \leq N} \left|\left(\partial^\alpha \big( \Gamma_T(g, g) + T(g, g, g) \big), \partial^\alpha g \right)\right|.
		\] 
		Based on Lemma \ref{Estimates of T} and Corollary \ref{Estimates of GammaT}, we deduce that
		\begin{align*} 
			&\sum_{1 \leq |\alpha| \leq N} \left|\left(\partial^\alpha \big( \Gamma_T(g, g) + T(g, g, g) \big), \partial^\alpha g \right)\right|\\[1mm]
			\lesssim&\sum_{\substack{|\alpha'|+|\alpha''|\le N  \\  1 \leq |\alpha| \leq N} }^{}\int_{\mathbb{R}^3}\left|\left\langle\Gamma_{{T}}(\partial^{\alpha'}g,\partial^{\alpha''}g),\partial^{\alpha}g\right\rangle\right|dx\\[1mm]
			&+\sum_{\substack{|\alpha'|+|\alpha''|+|\alpha'''|\le N  \\  1 \leq |\alpha| \leq N} }^{}\int_{\mathbb{R}^3}\left|\left\langle T(\partial^{\alpha'}g,\partial^{\alpha''}g,\partial^{\alpha'''}g),\partial^{\alpha}g\right\rangle\right|dx\\[1mm]
			\lesssim&\sum_{\substack{|\alpha'|+|\alpha''|\le N  \\  1 \leq |\alpha| \leq N} }^{}\int_{\mathbb{R}^3}\|\partial^{\alpha'}g\|_{L^{2}_{v}}\|\partial^{\alpha''}g\|_{L^{2}_{\mathcal{D}}}\|\partial^{\alpha}g\|_{L^{2}_{\mathcal{D}}}dx\\[1mm]
			&+\sum_{\substack{|\alpha'|+|\alpha''|+|\alpha'''|\le N  \\  1 \leq |\alpha| \leq N} }^{}\int_{\mathbb{R}^3}\|\partial^{\alpha'}g\|_{L^{2}_{v}}\|\partial^{\alpha''}g\|_{L^{2}_{v}}\|\partial^{\alpha'''}g\|_{L^{2}_{\mathcal{D}}}\|\partial^{\alpha}g\|_{L^{2}_{\mathcal{D}}}dx\\[1mm]
			\lesssim&\sum_{|\alpha'|+|\alpha''|\le N }^{} \left\|\|\partial^{\alpha'}g\|_{L^{2}_{v}}\|\partial^{\alpha''}g\|_{L^{2}_{\mathcal{D}}}\right\|_{L^{2}_{x}}\sqrt{\mathcal{D}_{N}(g(t))} \\[1mm]
			&+\sum_{|\alpha'|+|\alpha''|+|\alpha'''|\le N }^{}\left\|\|\partial^{\alpha'}g\|_{L^{2}_{v}}\|\partial^{\alpha''}g\|_{L^{2}_{v}}\|\partial^{\alpha'''}g\|_{L^{2}_{\mathcal{D}}}\right\|_{L^{2}_{x}}\sqrt{\mathcal{D}_{N}(g(t))}
		\end{align*}
	    Following the case analysis framework analogous to the proof of Lemma \ref{Macro estimate}, we employ the Sobolev inequalities  in Lemma \ref{Sobolev inequalities} to derive that
	\begin{align*} 
		  &\sum_{|\alpha'|+|\alpha''|\le N }^{} \left\|\|\partial^{\alpha'}g\|_{L^{2}_{v}}\|\partial^{\alpha''}g\|_{L^{2}_{\mathcal{D}}}\right\|_{L^{2}_{x}} \\[1mm] 
		\lesssim&\sum_{|\alpha'|=N }^{} \left\|\|\partial^{\alpha'}g\|_{L^{2}_{v}}\| g\|_{L^{2}_{\mathcal{D}}}\right\|_{L^{2}_{x}}+\sum_{|\alpha''|=N }^{} \left\|\| g\|_{L^{2}_{v}}\|\partial^{\alpha''}g\|_{L^{2}_{\mathcal{D}}}\right\|_{L^{2}_{x}}\\[1mm] 
		&+\sum_{|\alpha'|,|\alpha''|\le N-1 }^{} \left\|\|\partial^{\alpha'}g\|_{L^{2}_{v}}\|\partial^{\alpha''}g\|_{L^{2}_{\mathcal{D}}}\right\|_{L^{2}_{x}}\\[1mm] 
		\lesssim&\sum_{|\alpha'|=N }^{} \|\partial^{\alpha'}g\|_{L^{2}_{x}L^{2}_{v}}\| g\|_{L^{\infty}_{x}L^{2}_{\mathcal{D}}}  +\| g\|_{L^{\infty}_{x}L^{2}_{v}}\sum_{|\alpha''|=N }^{}\|\partial^{\alpha''}g\|_{L^{2}_{x}L^{2}_{\mathcal{D}}}  \\[1mm] 
		&+\sum_{|\alpha'|,|\alpha''|\le N-1 }^{}\|\partial^{\alpha'}g\|_{L^{3}_{x}L^{2}_{v}}\|\partial^{\alpha''}g\|_{L^{6}_{x}L^{2}_{\mathcal{D}}} \\[1mm] 
		\lesssim& \| g\|_{H^{N}_{x}L^{2}_{v}}\| \nabla_xg\|_{H^{1}_{x}L^{2}_{\mathcal{D}}}+\| g\|_{H^{2}_{x}L^{2}_{v}}\|  g\|_{\dot{H}^{N}_{x}L^{2}_{\mathcal{D}}}+ \| g\|_{H^{N}_{x}L^{2}_{v}}\| \nabla_xg\|_{H^{N-1}_{x}L^{2}_{\mathcal{D}}} \\[1mm] 
		\lesssim&\|g(t)\|_{H^{N}_{x}L^{2}_{v}}\sqrt{\mathcal{D}_{N}(g(t))},
	\end{align*}
and
    \begin{align*} 
    	&\sum_{|\alpha'|+|\alpha''|+|\alpha'''|\le N }^{}\left\|\|\partial^{\alpha'}g\|_{L^{2}_{v}}\|\partial^{\alpha''}g\|_{L^{2}_{v}}\|\partial^{\alpha'''}g\|_{L^{2}_{\mathcal{D}}}\right\|_{L^{2}_{x}} \nonumber\\[1mm] 
    	\lesssim&\sum_{|\alpha'|=N }^{} \left\|\|\partial^{\alpha'}g\|_{L^{2}_{v}}\|g\|_{L^{2}_{v}}\| g\|_{L^{2}_{\mathcal{D}}}\right\|_{L^{2}_{x}}+\sum_{|\alpha'''|=N }^{} \left\|\| g\|_{L^{2}_{v}}\| g\|_{L^{2}_{v}}\|\partial^{\alpha'''}g\|_{L^{2}_{\mathcal{D}}}\right\|_{L^{2}_{x}}\nonumber\\[1mm] 
    	&+\sum_{\substack{|\alpha''|,|\alpha'''|\le N-1  \\ |\alpha'| \leq N-2} }^{} \left\|\|\partial^{\alpha'}g\|_{L^{2}_{v}}\|\partial^{\alpha''}g\|_{L^{2}_{v}}\|\partial^{\alpha'''}g\|_{L^{2}_{\mathcal{D}}}\right\|_{L^{2}_{x}}\nonumber\\[1mm] 
    	&+\sum_{\substack{|\alpha'|,|\alpha''|\le N-1  \\ |\alpha'''| \leq N-2} }^{} \left\|\|\partial^{\alpha'}g\|_{L^{2}_{v}}\|\partial^{\alpha''}g\|_{L^{2}_{v}}\|\partial^{\alpha'''}g\|_{L^{2}_{\mathcal{D}}}\right\|_{L^{2}_{x}}\nonumber\\[1mm] 
    	\lesssim&\sum_{|\alpha'|=N }^{} \|\partial^{\alpha'}g\|_{L^{2}_{x}L^{2}_{v}}\| g\|_{L^{\infty}_{x}L^{2}_{v}}\| g\|_{L^{\infty}_{x}L^{2}_{\mathcal{D}}}  +\| g\|^{2}_{L^{\infty}_{x}L^{2}_{v}}\sum_{|\alpha'''|=N }^{}\|\partial^{\alpha'''}g\|_{L^{2}_{x}L^{2}_{\mathcal{D}}}  \nonumber\\[1mm] 
    	&+\sum_{|\alpha'| \le N-2 }^{}\|\partial^{\alpha '}g\|_{L^{\infty}_{x}L^{2}_{v}}\sum_{|\alpha''|,|\alpha'''|\le N-1 }^{}\|\partial^{\alpha''}g\|_{L^{3}_{x}L^{2}_{v}}\|\partial^{\alpha'''}g\|_{L^{6}_{x}L^{2}_{\mathcal{D}}} \nonumber\\[1mm] 
    	&+\sum_{|\alpha'|,|\alpha''|\le N-1 }^{}\|\partial^{\alpha'}g\|_{L^{3}_{x}L^{2}_{v}}\|\partial^{\alpha''}g\|_{L^{6}_{x}L^{2}_{v}}\sum_{|\alpha'''| \le N-2 }^{}\|\partial^{\alpha'''}g\|_{L^{\infty}_{x}L^{2}_{\mathcal{D}}} \nonumber\\[1mm] 
    	\lesssim& \| g\|_{H^{N}_{x}L^{2}_{v}}\| g\|_{H^{2}_{x}L^{2}_{v}}  \| \nabla_xg\|_{H^{1}_{x}L^{2}_{\mathcal{D}}}+\| g\|^{2}_{H^{2}_{x}L^{2}_{v}}\|  g\|_{\dot{H}^{N}_{x}L^{2}_{\mathcal{D}}}+ \| g\|^{2}_{H^{N}_{x}L^{2}_{v}}\| \nabla_xg\|_{H^{N-1}_{x}L^{2}_{\mathcal{D}}} \nonumber\\[1mm]  
    	\lesssim&\|g(t)\|^{2}_{H^{N}_{x}L^{2}_{v}}\sqrt{\mathcal{D}_{N}(g(t))}. \nonumber
    \end{align*}
	In summary, we obtain
	\small
	\begin{equation}\label{Ternary nonlinear estimate}
		\sum_{ |\alpha| \leq N} \left|\left(\partial^\alpha \big( \Gamma_T(g, g) + T(g, g, g) \big), \partial^\alpha g \right)\right|\lesssim \left(\|g(t)\|_{H^{N}_{x}L^{2}_{v}}+\|g(t)\|^{2}_{H^{N}_{x}L^{2}_{v}}\right) {\mathcal{D}_{N}(g(t))}. 
	\end{equation}
	\normalsize
	Combining Proposition \ref{prop:L_properties} with \eqref{energy identity}, \eqref{Binary nonlinear estimate}, and \eqref{Ternary nonlinear estimate}, we conclude that
	\small
	\begin{equation}\label{energy estimate 1}
		 \frac{1}{2} \frac{d}{dt} \| g(t) \|^2_{H^N_x L^2_v } +\delta_{0}\|g_{2}(t)\|^{2}_{H^{N}_{x}L^{2}_{\mathcal{D}}} \lesssim \left(\|g(t)\|_{H^{N}_{x}L^{2}_{v}}+\|g(t)\|^{2}_{H^{N}_{x}L^{2}_{v}}\right) {\mathcal{D}_{N}(g(t))}. 
	\end{equation}
	\normalsize	
	We now choose a positive constant $\varepsilon$ such that $\varepsilon < \min\left\{ \frac{1}{4C_0}, \frac{\delta_0}{4C_1} \right\}$, where the constants $C_0$ and $C_1$ are specified in Lemma \ref{Macro estimate}. 
    Subsequently, by combining equation \eqref{energy estimate 1} with $\varepsilon\times$\eqref{Macro estimate 1}, we can identify a constant $\delta_{1} > 0$ such that 
	\begin{align}\label{energy estimate 2}
		\begin{aligned}
			\frac{1}{2} \frac{d}{dt} &\left(\| g(t) \|^2_{H^N_x L^2_v }+2\varepsilon\mathcal{I}(t)\right)+\delta_{1}\left(\|\nabla_x[a,\mathbf{b},c](t)\|^{2}_{H^{N-1}_{x}}+\|g_{2}(t)\|^{2}_{H^{N}_{x}L^{2}_{\mathcal{D}}} \right)\\  
			\lesssim& \left(\|g(t)\|_{H^{N}_{x}L^{2}_{v}}+\|g(t)\|^{2}_{H^{N}_{x}L^{2}_{v}}+\|g(t)\|^{4}_{H^{N}_{x}L^{2}_{v}}\right) {\mathcal{D}_{N}(g(t))}.
		\end{aligned} 
	\end{align} 
	\normalsize
	We proceed to define 
	\[\mathcal{E}_{N}(g(t)):=\| g(t) \|^2_{H^N_x L^2_v }+2\varepsilon\mathcal{I}(t).\]
	Observe that the selection of \( \varepsilon \) ensures  
	\[\frac{1}{2}\| g(t) \|^2_{H^N_x L^2_v }\le\mathcal{E}_{N}(g(t))\le\frac{3}{2}\| g(t) \|^2_{H^N_x L^2_v }.\]
	Consequently, by synthesizing equation \eqref{energy estimate 2} with the definition of the dissipation functional \( \mathcal{D}_N \) in \eqref{dissipation rate}, we deduce the existence of a constant \( C_2 > 0 \) such that the following inequality holds: 
	\small  
	 \begin{align}\label{energy estimate 3} 
	 	\begin{aligned}
	 		\frac{1}{2} \frac{d}{dt}  \mathcal{E}_{N}(g(t))+\delta_{1}\mathcal{D}_{N}(g(t))\le C_2 \left(\sqrt{\mathcal{E}_{N}(g(t))}+\mathcal{E}_{N}(g(t))+\mathcal{E}^{2}_{N}(g(t))\right) {\mathcal{D}_{N}(g(t))}.
	 	\end{aligned} 
	 \end{align}  
 \normalsize
 
 Finally, we prove global existence by utilizing the standard bootstrap argument.  
 For this purpose, we choose 
 $ 0 < M \le \min \{{M_0},  \frac{\delta_{1}^{2}}{54C_{2}^{2}} ,\frac{2}{3}\}   ,$
 where \( M_0 \) is the constant  given in Theorem \ref{local existence}.
 Let $T_{\max}$ be the maximal existence time of the local solution. We define 
 \begin{equation}\label{bootstrap assumption}
 	T_{\text{boot}}:=\sup\left\{0<T<T_{\max}:\sup_{0\le t\le T} \| g(t) \|^2_{H^N_x L^2_v }\le M\right\}. 
 \end{equation} 
 Now, we choose \( \epsilon \) in Theorem \ref{main_theorem} such that
 $ 0 < \epsilon \le \min \{\frac{M}{4},1\}.$
 Furthermore, from Theorem \ref{local existence}, we know that \( T_{\text{boot}} > 0 \) and that the following estimates hold for $0\le t<T_{\text{boot}}$.
 By \eqref{energy estimate 3} and the choice of \( M\), it is straightforward to infer that
 \begin{equation}\label{L2-energy-estimate}
 	\frac{d}{dt}\mathcal{E}_{N}(g(t))+ \delta_{1}\mathcal{D}_{N}(g(t)) \le0, 
 \end{equation}  
 For every $0<T<T_{\text{boot}}$, integration yields
 $$\sup_{0\le t\le T} \| g(t) \|^2_{H^N_x L^2_v }\le2\sup_{[0,T]}\mathcal{E}_{N}(g(t)) \le2\mathcal{E}_{N}(g_{0})\le 3 \| g_{0} \|^2_{H^N_x L^2_v }\le 3\epsilon<M.$$ 
	The strict improvement of the bootstrap bound and the uniform local existence time in Theorem \ref{local existence} exclude a finite bootstrap or maximal existence time. Thus $T_{\text{boot}}=T_{\max}=\infty$. Integrating \eqref{L2-energy-estimate} also gives the asserted bound on the time integral of $\mathcal{D}_N$. This establishes the existence of a global solution to the Cauchy problem \eqref{BT-g}, thereby completing the proof of Theorem \ref{main_theorem}.

		\begin{appendices}
			\section{Appendix}\label{sec;app}

		   \paragraph{A.1. Properties of the symmetrized ternary Boltzmann opeartor $Q_{T}$.}\label{A.1.}
		 
	This subsection summarizes several key properties of the ternary Boltzmann collision operator \(Q_T\), which parallel those of the binary operator \(Q_B\). These results are reproduced directly from Section 5.4 of  \cite{MR4337060}, where complete proofs and detailed explanations can be found.
		  
		   \begin{proposition}[Weak Formulation and Entropy Dissipation]\label{prop:entropy}
		   	Let \( f: [0, \infty) \times \mathbb{R}^6 \to \mathbb{R} \). The following properties hold:
		   	\begin{enumerate}
		   		\item[(\romannumeral 1)] For any test function \( \phi: [0, \infty) \times \mathbb{R}^6 \to \mathbb{R} \), the weak formulation of the ternary operator \( Q_{T} \) is given by:
		   		\begin{align*}
		   			\int_{\mathbb{R}^3} Q_{T}(f, f, f) \phi \, dv = \frac{1}{2} &\int_{S^5 \times \mathbb{R}^9} B_3(\bm{u}, \bm{\omega}) \left( f^* f_1^* f_2^* - f f_1 f_2 \right) \\
		   			&\quad\times(\phi + \phi_1 + \phi_2 - \phi^* - \phi_1^* - \phi_2^*) \, d\bm{\omega} dv_{1}dv_{2} dv. 
		   		\end{align*}  
		   		\item[(\romannumeral 2)] If \( f > 0 \), the entropy dissipation satisfies:
		   		\[
		   		D_{T}(f) := \int_{\mathbb{R}^3} Q_{T}(f, f, f) \ln f \, dv \leq 0.
		   		\]  
		   	\end{enumerate}      
		   \end{proposition}

		   \begin{proposition}[Characterization of Ternary Collision Invariants]\label{prop:collision_invariant}
		   	Let \( \phi: \mathbb{R}^3 \to \mathbb{R} \) be a continuous function. Then \( \phi \) is a collision invariant for ternary interactions, i.e.,  
		   	\[
		   	 \phi(v) + \phi(v_1) + \phi(v_2)=\phi(v^*) + \phi(v_1^*) + \phi(v_2^*),  \quad \forall \, (\omega_1, \omega_2, v, v_1, v_2) \in \mathbb{S}^{5} \times \mathbb{R}^{9},
		   	\]  
		   	if and only if \( \phi \) has the form:  
		   	\[
		   	\phi(v) = a +  \mathbf{b}\cdot v  + c|v|^2, \quad v \in \mathbb{R}^3,
		   	\]  
		   	where \( a, c \in \mathbb{R} \) and \( \mathbf{b} \in \mathbb{R}^3 \). 
		   \end{proposition}

		   \begin{proposition}[Characterization of Equilibrium]\label{prop:equilibrium}
		   	Let \( f: [0, \infty) \times \mathbb{R}^6 \to (0, \infty) \) be positive and velocity-decreasing. Then:
		   	\begin{enumerate}
		   		\item[(\romannumeral 1)] \( D_{T}(f) = 0 \) if and only if \( f \) is a Maxwellian of the form:
		   		\[
		   		f(t, x, v) = \frac{R(t, x)}{(2\pi T(t, x))^{3/2}} e^{-\frac{|v - U(t, x)|^2}{2T(t, x)}},
		   		\]
		   		where \( R, T: [0, \infty) \times \mathbb{R}^3 \to (0, \infty) \) and \( U: [0, \infty) \times \mathbb{R}^3 \to \mathbb{R}^3 \) are continuous functions.
		   		\item[(\romannumeral 2)] \( Q_{T}(f, f, f) = 0 \) if and only if \( f \) is a Maxwellian.
		   	\end{enumerate}       
		   \end{proposition}
		   
		   \begin{proposition}[Local Conservation Laws]\label{prop:conservation}
		   	Let \( f \) be a solution to the symmetrized ternary Boltzmann equation, i.e., 
		   	$$\partial_t f + v \cdot \nabla_x f =Q_{T}(f,f,f). $$
		   	The following local conservation laws hold:
		   	\begin{enumerate}
		   		\item[(\romannumeral 1)] {Mass conservation}:
		   		\[
		   		\partial_t \int_{\mathbb{R}^3} f \, dv + \nabla_x \cdot \int_{\mathbb{R}^3} f v \, dv = 0.
		   		\]
		   		\item[(\romannumeral 2)] {Momentum conservation}:
		   		\[
		   		\partial_t \int_{\mathbb{R}^3} v f \, dv + \nabla_x \cdot \int_{\mathbb{R}^3} f \, v \otimes v \, dv = 0. 
		   		\]
		   		\item[(\romannumeral 3)] {Energy conservation}:
		   		\[
		   		\partial_t \int_{\mathbb{R}^3} |v|^2 f \, dv + \nabla_x \cdot \int_{\mathbb{R}^3} |v|^2 f v \, dv = 0. 
		   		\]
		   	\end{enumerate}     
		   \end{proposition}

		  \paragraph{A.2. Convolution estimates.}

		   \begin{lemma}\label{Convolution estimates}
		     	Let \( d \) be a positive integer. For \( z \in \mathbb{R}^d \), \( q > -d \), and \( \rho > 0 \), the following asymptotic equivalence holds:  
		     	\[
		     	\int_{\mathbb{R}^d} |z - y|^q e^{-\rho |y|^2} \, dy \sim \langle z \rangle^q,
		     	\]  
		     	where \( \langle z \rangle := \sqrt{1 + |z|^2} \). 
		   \end{lemma}
		   \begin{proof}  
		   	The case for \( q > 0 \) can be derived analogously to the proof for \( q = 1,\ d=3 \) in (3.54) of \cite{Glassey1996}, thus we focus here on the case where \( q \in (-d, 0] \).
		   		We analyze the integral by decomposing the domain into two regions: \( |z - y| \leq 1 \) and \( |z - y| > 1 \). The proof proceeds in two directions to establish upper and lower bounds.
		   		
		   		\noindent \textit{Upper Bound.}  
		   		Split the integral:
		   		\[
		   		\int_{\mathbb{R}^d} |z - y|^q e^{-\rho |y|^2} \, dy = \underbrace{\int_{|z - y| \leq 1} |z - y|^q e^{-\rho |y|^2} \, dy}_{I_1} + \underbrace{\int_{|z - y| > 1} |z - y|^q e^{-\rho |y|^2} \, dy}_{I_2}.
		   		\]
		   		For \( I_1 \), observe that \( |z - y| \leq 1 \) implies \( |y| \geq |z| - 1 \). If \( |z| \ge 2 \), then \( |y| \geq \frac{|z|}{2} \) by the reverse triangle inequality. Consequently:
		   		\[
		   		e^{-\rho |y|^2} \leq e^{-\frac{\rho}{4} |z|^2}.
		   		\]
		   		Thus,
		   		\[
		   		I_1 \leq e^{-\frac{\rho}{4} |z|^2} \int_{|z - y| \leq 1} |z - y|^q \, dy \lesssim e^{-\frac{\rho}{4} |z|^2} \lesssim \langle z \rangle^q,
		   		\]
		   		where the last inequality follows from the exponential decay dominating polynomial growth.
		   		The case where \(|z| \leq 2\) is trivial: the inequality \(I_1 \lesssim 1 \lesssim\langle z\rangle^q\) holds directly.
		   		For \( I_2 \), utilize the inequality \( \langle z \rangle \lesssim \langle z - y \rangle \langle y \rangle \lesssim |z - y| \langle y \rangle \), valid for \( |z - y| \geq 1 \). This implies:
		   		\[
		   		|z - y|^q \lesssim \frac{\langle z \rangle^q}{\langle y \rangle^q}.
		   		\]
		   		Therefore,
		   		\[
		   		I_2 \lesssim \langle z \rangle^q \int_{\mathbb{R}^d} \langle y \rangle^{-q} e^{-\rho |y|^2} \, dy \lesssim \langle z \rangle^q,
		   		\]
		   		since \( \langle y \rangle^{-q} e^{-\rho |y|^2} \) is integrable for \( q > -d \).
		   		Combining \( I_1 \) and \( I_2 \):
		   		\[
		   		\int_{\mathbb{R}^d} |z - y|^q e^{-\rho |y|^2} \, dy \lesssim \langle z \rangle^q.
		   		\]
		   		
		   		\noindent \textit{Lower Bound.}  
		   		Consider the region \( |y| \leq 1 \). Here, \( |z - y| \leq |z| + 1 \), and since \( q \leq 0 \):
		   		\[
		   		|z - y|^q \geq (|z| + 1)^q \gtrsim \langle z \rangle^q.
		   		\]
		   		Thus,
		   		\[
		   		\int_{\mathbb{R}^d} |z - y|^q e^{-\rho |y|^2} \, dy \geq \int_{|y| \leq 1} |z - y|^q e^{-\rho |y|^2} \, dy \gtrsim \langle z \rangle^q \int_{|y| \leq 1} e^{-\rho |y|^2} \, dy \gtrsim \langle z \rangle^q.
		   		\]
		   		
		   		\noindent The upper and lower bounds together yield:
		   		\[
		   		\int_{\mathbb{R}^d} |z - y|^q e^{-\rho |y|^2} \, dy \sim \langle z \rangle^q.
		   		\]

		   \end{proof}
	   
	    \paragraph{A.3. Sobolev inequalities.}
	    
	    We state some standard Sobolev inequalities that are frequently used in this paper.

		\begin{lemma}\label{Sobolev inequalities}
			Let \( u \in H^2(\mathbb{R}^3) \). Then
			\begin{enumerate}
				\item [(i)] 
				$ \|u\|_{L^\infty} \lesssim\|\nabla u\|^{\frac{1}{2}}_{L^2} \|\nabla^2 u\|^{\frac{1}{2}}_{L^2} \lesssim\|\nabla u\|_{H^1}; $ 
				\item [(ii)] 
				$ \|u\|_{L^6} \lesssim \|\nabla u\|_{L^2}; $ 
				\item [(iii)] 
				$ \|u\|_{L^q} \lesssim \|u\|_{H^1}, \quad 2 \leq q \leq 6.  $
			\end{enumerate}
		\end{lemma}
		
	    \end{appendices}

		
		
    	\noindent{\bf Acknowledgements.} \   The authors would like to thank the anonymous referees for helpful comments and suggestions to improve the quality of redaction. The research of Linjie Xiong is supported by the National Natural Science Foundation of China (Grant No. 12231006) and the Fundamental Research Funds for the Central Universities.

		

		

	\end{document}